\documentclass{article}
\usepackage{tikz-network}
\usepackage[T1]{fontenc}
\usepackage{lmodern}
\usepackage{stackengine}

\usepackage[greek, english]{babel}
\usepackage{amsbsy,amssymb,amscd,amsmath,amsthm,amsfonts,mathrsfs}
\usepackage{tikz}
\usetikzlibrary{graphs,quotes,arrows.meta}
\usepackage[margin=1in, top=1.5in]{geometry}
\usepackage{amsmath,amssymb,bbm}
\usepackage{graphicx}
\usetikzlibrary{arrows, positioning}
\usepackage{tikz}
\usetikzlibrary{arrows.meta}
\usepackage{verbatim}
\usepackage{caption}
\usepackage{subcaption}
\newcounter{IssueCounter}
\newtheorem{Issue}[IssueCounter]{Issue}

\usepackage{lineno}
\def\g{x}
\def\blockX{X}
\def\blockY{Y}
\def\blockV{V}

\def\ad{\mathop{\rm ad}\nolimits}

\newfont{\gothic}{eufm10}
\def\bas {\begin{eqnarray*}}
	\def\eas {\end{eqnarray*}}

\def\N{\mathsf{N}}

\def\M{\mathsf{M}}
\def\H{\mathsf{H}}

\def\Sl{{\mathfrak{sl}_2}}
\def\ba {\begin{eqnarray}}
\def\ea {\end{eqnarray}}
\newtheorem{theorem}{Theorem}[section]
\newtheorem{lemma}[theorem]{Lemma}
\newtheorem{proposition}[theorem]{Proposition}
\newtheorem{corollary}[theorem]{Corollary}
\newtheorem{definition1}[theorem]{Definition}

\newtheorem{remark1}[theorem]{Remark}
\newenvironment{remark}{\begin{remark1}\rm}{\end{remark1}}

\newtheorem{example1}[theorem]{Example}

\def\barray{\begin{eqnarray*}}             \def\earray{\end{eqnarray*}}
\def\beq{\begin{equation}} \def\eeq{\end{equation}}

\title{Nilpotent 
	Feed Forward Network Dynamics}
\author{Fahimeh Mokhtari \thanks{Department of Mathematics, Vrije Universiteit Amsterdam, The Netherlands, {\tt f.mokhtari@vu.nl}.} }
\date{}

\providecommand{\keywords}[1]{\textbf{\textit{Keywords---}} #1}

\begin{document}

\date{}
	\maketitle
		\rule{13cm}{0.03cm} 
	\begin{abstract}
	In this paper, we explore the normal form of fully inhomogeneous feed forward network dynamical systems, characterized by a nilpotent linear component. We introduce a new normal form method, termed the {\em triangular 
	$\Sl$-style}, to categorize this normal form.

 We aim to establish a comprehensive normal form theory, extending to quadratic terms across all dimensions. To accomplish this, we leverage mathematical tools including Hermite reciprocity, transvectants, and insights derived from the well-known $3$-dimensional normal form related to Sylvester's work on generating functions for quadratic covariants.
	
	Furthermore, we formulate the orbital normal form in a general Lie algebraic context as the result of an outer transformation
 and expand it into {\em block-triangular} outer normal form. This extended framework is subsequently employed in both $2$D and $3$D scenarios, leading to noteworthy simplifications that prepare these systems for the study of bifurcations.

	\end{abstract}
		\tableofcontents
	\section{Introduction}

 Network dynamical systems are prevalent in various fields, including chemistry, computer science, biology, engineering, neural networks, and more, as evidenced in \cite{subotic2020lyapunov,majdandzic2014spontaneous,hopfield1982neural}. This class of dynamical systems has garnered attention from mathematicians 
\cite{rink2015coupled,rink2013amplified,von2022amplified,golubitsky2012feed,golubitsky2004some,field2004combinatorial,golubitsky2005synchrony,stewart2004networking,stewart2003symmetry,rink2014coupled,golubitsky2006nonlinear,vassena2024structural,vassena2024finding,vassena2024unstable} due to its wide-ranging applications and the need to understand network dynamics. Many works have focused on extending classical concepts to analyze the dynamical behavior of systems with a network structure, encompassing center manifold reduction, normal forms, bifurcation, synchrony, and symmetries.

In this paper, we present a novel study and machinery for the normal form of feed forward networks, one of the first explorations in this direction. Feed forward networks are characterized by loop-free node connections, and cells receive input only from the cells below.  Specifically, we delve into a class of nilpotent feed forward network dynamical systems, aiming to provide deeper insights into their local dynamics and properties.
	
Our focus lies on the following class of nilpotent feed forward network dynamical systems:
 \begin{align}\label{eq:feed}
      \begin{cases} 
{\dot{x}_1}&= X_1={x_{2}}+f_1({x_{1}},{x_{2}},\cdots, x_n),
	\\
{\dot{x}_2}&=X_2={x_{3}}+f_2({x_{2}},\cdots, x_n),
	\\
	&\vdots
	\\
	{\dot{x}_{n-1}}&=X_{n-1}=x_{n}+f_{n-1}(x_{n-1},x_n),
	\\
	{\dot{x}_n}&= X_n=f_n(x_n),
	\end{cases}
 \end{align}
where \(\frac{\partial f_i}{\partial x_j}(0,\cdots,0)=0\) 
 for all \(1 \leq i,j \leq n.\) Here we assume that the feed forward architecture 
	has a single node, and each variable corresponds with a cell.

 To analyze the local behavior of dynamical systems near a singular point at the origin, we employ the nonlinear normal form theory \cite{murdock2006normal,sanders2007averaging,cushman1984normal,qin2022high}. 
 See also \cite{tall2002feedback,brockett1978feedback}, where the authors study the normal form of these systems with control input.

 Our goal in this study is to normalize the network \eqref{eq:feed} at the origin, taking into account its nilpotent linear part. Although there exists a vast theory of nilpotent normal forms \cite{sanders2002spectral,sanders2007averaging,cushman56nilpotent,cushman1984normal,cushman1988normal,murdock2006normal}, its immediate application is hampered by the specific architecture of the network. We will explain this in the motivational example; however, before doing so, we will recall the 
 \(\Sl\)-style of normal form. 

 The $\Sl$-style of normal form for vector fields was introduced by Cushman and Sanders \cite{cushman56nilpotent}, with historical details available in \cite{cushman1990survey}. Consider a given dynamical system $\dot{x}=\N x+f(x)$ with $x\in \mathbb{R}^n$ and nilpotent $\N$. The Jacobson–Morozov theorem guarantees the existence of a nilpotent $\M$ and semi-simple operator $\H$ such that $\langle \N,\H,\M \rangle$ forms a $\Sl$ triple in \(\mathfrak{gl}_n(\mathbb{R})\).
Recall that 
   \(\Sl\)-triple is generated  by \(\langle\N,\H,\M\rangle\)  with the  commutator relations \([\H, \N]=-2\N, 
 [\H, \M]=2 \M,\) and \([\M, \N]= \H\).
The normal form is generated by $\ker{\rm ad} (\M)$, where ${\rm ad}(x)y=[x,y].$ See, \cite{baider1992further,sanders2002spectral,cushman1988normal,cushman1988splitting,cushman56nilpotent}. For the nilpotent normal form classification $2$D and $3$D, see \cite{baider1992further,baider1991unique,gazor2019vector}. Additional results on this topic are available in \cite{gazor2014normal,gazor2013volume}, where the authors classify the simplest normal form for a Hopf-zero singularity in the presence of a nilpotent term in the second-level normal form.

Obtaining the normal form of a given dynamical system involves applying local formal coordinate transformations. Therefore, we must determine the Lie algebraic structure of the transformation generators that preserve the system's architecture. 
It's important to note that the choice of preserving the structure after the normal form or using a larger transformation space without preserving the network structure is subjective. For the network considered here, we aim to preserve the feed forward structure after applying the normal form transformations. 

\noindent{\bf{Motivational example}}
To illustrate this concept in normal form $\Sl$ style, consider the example of a 2D feed forward network:
\begin{equation}
\label{eq:intro1}
\left\{
\begin{array}{l} 
    \dot{x}_{1} = x_{2} + F_1(x_{1}, x_{2}),\\ 
    \dot{x}_{2} = F_2(x_{2}),
\end{array}
\right.
\end{equation}
with  $\frac{\partial }{\partial x_i}F_1({x_{1}},{x_{2}})=0$ and $\frac{\partial }{\partial x_2}F_2({x_{2}})=0$ at  the origin  for $i=1,2.$ The  coordinate transformation should keep the change in ${x_{2}}$ independent of the ${x_{1}}$ coordinate to preserve the feed forward structure.

Consider the following nilpotent system $2$D:\ba\label{eq:Gn2}
	\left\{ \begin{array}{l} \dot {x}_{1} = 
		{x_{2}}+ f_1({x_{1}},{x_{2}}),\\ 
		\dot {x}_{2} =f_2({x_{1}},{x_{2}}),
	\end{array} \right.	\ea
		where \(\frac{\partial f_1}{\partial x_i}({x_{1}},{x_{2}})=\frac{\partial f_2}{\partial x_i}({x_{1}},{x_{2}})=0\) at origin for \(i=1,2\) are nonlinear polynomial vector fields, and \(\N={x_{2}} \frac{\partial}{\partial {x_{1}}}\). 
  One can embed \(\N\) into  an  \(\Sl\)-triple, 
	 \(\langle \N,\H,\M \rangle\)  as
	\(\M={x_{1}} \frac{\partial}{\partial {x_{2}}}\) and \(
	\H={x_{1}} \frac{\partial}{\partial {x_{1}}}-{x_{2}} \frac{\partial}{\partial {x_{2}}}.\)

 Therefore, the general normal form of \eqref{eq:Gn2} is given by 
	\bas
	\left\{ \begin{array}{l} \dot {x}_{1} = 
		{x_{2}}+ {x_{1}}g({x_{1}}),\\ 
		\dot {x}_{2} ={x_{2}}g({x_{1}})+f({x_{1}}),
	\end{array} \right.
	\eas 
 where the nonlinear terms are in \(\ker\ad{\M}\), according to the definition of  $\Sl$-style.
When we look at the feed forward structure, we see that the normal form described by Equation \eqref{eq:intro1} disappears. This means that the usual way of representing the normal form for nilpotent singularities, the
$\Sl$-style does not work for this type of network.  The reason for this is that neither 
$\M$
 nor its kernel fit into the Lie algebra of feed forward networks. We will discuss this in more detail in Section \ref{sec:algorithm}.

Now, the question arises: how can we establish a normal form for these feed forward networks? Surprisingly, we can use the  \(\Sl\)-style for functions (instead of vector fields) to define such a normal form, which we refer to as the triangular \(\Sl\)-style. Specifically, the normal form  in the \(2\)D case, Equation \eqref{eq:intro1},
using this style will be as follows:
 
		\bas
\left\{ \begin{array}{l} \dot {x}_{1} = 
	{x_{2}}+ \bar{f}_1({x_{1}}),\\ 
	\dot {x}_{2} =F_2({x_{2}}).
\end{array} \right.
	\eas

 \noindent{\bf What do we call the triangular 
$\Sl$-style?}
The space of a feed forward network forms a Lie algebra, as demonstrated in Proposition \ref{pro:lie}. This means that the Lie product of two forward vector fields of the feed forward, expressed in \eqref{eq:feed}, remains in this form. In addition to the Lie algebra properties, the Lie bracket of each component of the feed forward network also exhibits feed forward (triangular) architectures. This means that the Lie product of a vector field of the form \eqref{eq:feed} in the \(p\)-th component with its linear part generates terms in both the \(p\)-th and \(p-1\)-th components.

Let us elaborate on this. Denote by $\mathfrak{\g}^p$  the \(p\)-th component of \eqref{eq:feed}. For given $f_i(x_i, \cdots, x_n) \frac{\partial}{\partial x_i} \in \mathfrak{\g}^i$ and $g_j(x_j, \cdots, x_n) \frac{\partial}{\partial x_j} \in \mathfrak{\g}^j$, we obtain
\[
[f_i(x_i, \cdots, x_n) \frac{\partial}{\partial x_i}, g_j(x_j, \cdots, x_n) \frac{\partial}{\partial x_j}]
\subseteq \mathfrak{\g}^{\min(i,j)}.
\]

See Proposition \ref{pro:lie} and Corollary \ref{cor:Lie}. 

To begin, we fix the \(p\)-th component and construct the \(\Sl\) around \(\N_p\), which is a \((n-p+1) \times (n-p+1)\) matrix obtained by eliminating the first \(p-1\) rows and columns of \(\N\). Refer to Figure \ref{fig:inductive} for an illustration. Then, for each component, we find an \(\Sl\)-triple \(\langle \N_{p},\H_{p},\M_{p} \rangle\), as outlined in Theorem \ref{thm:sl2triple}. In every instance, the normal form is defined as those terms belonging to the  \(\ker\mathsf{L}_{\M_{p}}\), where the \(\mathsf{L}\) operator is defined in Theorem \ref{thm:Lop}. In summary, for each component, we can construct a \(\Sl\) triple, and the terms in the  \(\ker\mathsf{L}_{\M_{p}}\) will determine the normal form for that component.

Apart from classifying normal forms up to any order, the triangular \(\Sl\)-style introduces an intriguing property: the normal form of a feed forward network of dimension \(n\) contains the normal form of a feed forward system of dimension \(n-1\). This property is illustrated with examples in Section \ref{subsec:examples}. This means that if we can describe the normal form of a system of \(n-1\) dimensional dimensions, this can be used without changing the description of the normal form of an \(n\) dimensional system with irreducible nilpotent in the normal form of Jordan (apart from renumbering the variables of \(x_1,\cdots,x_{n-1}\) to \(x_2,\cdots,x_n\)).

The next question to be addressed is how to find the  \(\ker \mathsf{L}_{\M_{p}}\). Generally, the classification of invariants can be challenging when the matrix dimension and the degree of the polynomial increase. However, complete  results can be found in \cite[Chapter 12]{sanders2007averaging} for low-dimensional irreducible cases from dimension \(2\) to \(5\). In this paper, we determine the quadratic kernel of \(\mathsf{L}_{\M_{p}}\) for every \(n \in \mathbb{N}\) by computing the transvectant of \({x_{1}}\) (which is the trivial kernel for this operator; see Equation \eqref{eq:inductivesl2}), as described in Section \ref{sec:inv}. This provides an explicit quadratic normal form for any dimension of the feed-forward network, sufficient for local studies of systems near steady-state solutions; see Theorem \ref{thm:ndnf}.

Furthermore, we determine the general normal form in the dimensions $2$ and $3$ (refer to Sections \ref{sec:2dorbital}--\ref{sec:3dorbital}). These general forms serve as the outer normal form, which will be discussed in the following paragraph.

Besides the method to find the normal form for feed forward networks, we propose an approach that utilizes Hermite reciprocity to test the quadratic normal form, Section \ref{SEC:NFAPP}. 

This method
allows us to apply the Cushman-Sanders test \cite{cushman56nilpotent}. By employing Hermite reciprocity, which establishes the isomorphism between the two representations \(S^m S^n\mathbb{C}^2\) and \(S^n S^m\mathbb{C}^2\) of \(\bf{GL(2)}\) \cite{Goldstein2011}, 
applied to the generating functions of covariants, we achieve the desired results. In addition, references \cite[Theorem 6.31]{olver1987invariant} and \cite{sturmfels2008algorithms,hermite1854} provide further insight into this technique. Notably, this is the first usage of Hermite reciprocity in the field of normal form theory as far as we know.

 \noindent{\bf How do we define the  {Outer normal form?}}

Another method that we have developed here is the {\em{Outer normal form}}, which is a form of orbital normal form. In the orbital normal form method, in addition to employing coordinate transformations, functions can be utilized as transformations, which allows for further reduction, \cite{gazor2013normal, bogdanov1979local, bogdanov1976local}.
For example, consider the orbital normal form of the Eulerian family of Hopf-zero singularities, where most terms, except for some leading terms, are removed \cite{gazor2013normal}.

    In this paper, we present a Lie algebraic formulation of orbital normal form theory, defining an extension of the transformations and demonstrating that there is a Lie algebraic representation of these transformations in the vector fields (refer to Theorem \ref{thm:oop}). Notably, this is the first time the orbital normal form is formulated in a Lie algebraic context, as far as we know. 
Furthermore, we extend the classical orbital normal form by combining it with the right-left equivariant and strongly equivariant methods \cite{golubitsky2012singularities}. The result introduces an outer normal form, applicable beyond feed forward network dynamical systems. However, in this paper, we focus exclusively on its application to the specific problem at hand.
We apply the outer normal form to \(2\)D and \(3\)D feed forward networks, eliminating almost all higher-order terms, except for quadratic terms.
The author believes that this method would work for the higher-order dimension.  Since it involves heavy computations we would leave it in this paper.
Moreover, we show that in general the quadratic terms can be reduced to \(x_i^2\frac{\partial}{\partial x_i}, i=1,\ldots,n\), see Theorem \ref{thm:ndonf}.

\begin{theorem}[Outer unique normal form]\label{thm:ONF3} 
For the {\em $3$D} feed forward network described by the following system near the triple zero singular fixed point {\em (}that is, without versal parameters in the linear part {\em )}:
 \begin{align*}
      \begin{cases} 
		\dot{x}_1 &= {x_{2}} + F_1({x_{1}},{x_{2}},{x_{3}}), \\
		\dot{x}_2 &= {x_{3}} +F_2({x_{2}},{x_{3}}), \\
		\dot{x}_3 &= F_3({x_{3}}),
	  \end{cases}
  \end{align*}
 there exists  a formal invertible feed forward transformation that brings it to the following outer  normal form:
 	\begin{align}\label{eq:3dl1}
	\begin{cases} 
		\dot {x}_{1} &= {x_{2}} + a_2 x_1^2, \\
		\dot {x}_{2} &= {x_{3}} + b_2x_{2}^2, \\
		\dot {x}_{3} &= c_2x_3^2,
	\end{cases}
\end{align}
assuming the coefficients \(a_2,b_2\) and \(c_2\) are invertible. 
\end{theorem} 

	\noindent
	{\bf Organization of paper}: 
 
We begin, in section \ref{sec:preli}, with the preliminaries, providing an overview of normal form theory and the associated challenges in determining their normal forms using $\Sl$ representations.

In section \ref{sec:algorithm}, we introduce the concept of the triangular $\Sl$-style normal form. This is a style for systematically classifying normal forms in feed forward networks. In section \ref{sec:inv} we cover the classification of quadratic normal forms and validate these normal forms in this section.

Section \ref{sec:3d} provides clear insight through a set of examples that illustrate key concepts. We explore the computational details of the quadratic normal form, with a specific focus on the $3$D scenario. In addition, we present a range of examples that span dimensions from $2$D to $8$D, demonstrating the theory developed.

In section \ref{sec:orbital}, we introduce the method of outer normal form and specifically apply this method to $2$D and $3$D feed forward networks.

	\section{Preliminaries on normal form}\label{sec:preli}
In this section, we aim to provide some preliminary information regarding normal form. Let's consider the following dynamical system:
\[
\dot{x} =  {\mathsf A}x + f(x),\quad x \in \mathbb{R}^n,
\]
where ${\mathsf A}$ is the linear part and $f(x)$ is the nonlinear vector field of the system.

We rewrite this in Lie algebraic terms so that $v^{[0]} = \sum_{i=0}^{\infty} v_i^{[0]}$ is given, where $v_i \in \mathcal{V}_i$ are the vector fields with grade $i$. 

In the first-level normal form, the grade is the \(x\)-degree minus one. In this paper, we use the standard notation $\delta^{(1)}$ for the grading of vector fields in the first-level normal form study. That is, \(\delta^{(1)}v_i=i\).

For the second level we use $\delta^{(2)}$, see Sections \ref{sec:2dorbital} and \ref{sec:3dorbital}. Now, define
\bas
&\mathcal{N}_i: \mathcal{V}_i \rightarrow \mathcal{V}_i&
 \\
 \,\,\,\,\,\,\,\qquad &g_i  \mapsto  [v_0^{[0]}, g_i],&
 \eas
where $[x,y] = {\mathsf{ad}}_x(y) := xy-yx$. 

Then there exists a complementary  space $\mathcal{C}_i$ such that $\mathcal{V}_i = \text{im}(\mathcal{N}_i) \oplus \mathcal{C}_i$.

Now, we define the normal form as
$
v^{[1]} = \sum_{i=0}^{\infty} w_i^{[1]},
$
where $w_i \in \mathcal{C}_i$, see \cite{sanders2002spectral, sanders2007averaging, murdock2006normal}. 
Given \(v^{[0]}_1\), we determine \(w_1^{[1]}\) and \(g^{[1]}_1\) by the decomposition: \(v^{[0]}_1=w_1^{[1]}+\mathcal{N}_1 g^{[1]}_1,
w_1^{[1]}\in\mathcal{C}_1\).
For instance, to find the normal form of the \(v_1\) we apply the transformation \(g_1\) to the vector field \(v^{[0]}\) 
and compute modulo grade \(2\) terms:
\ba
\nonumber
\exp({\mathsf{ad}}_{g_1^{[1]}})v^{[0]}&=&v^{[0]}_0+v^{[0]}_1-[v^{[0]}_0,g_1^{[1]}]+\cdots
\nonumber
\\&=&v^{[0]}_0+v^{[0]}_1-\mathcal{N}_1 g^{[1]}_1+\cdots
\nonumber
\\&=&v^{[0]}_0+w_1^{[1]}+\cdots
\label{eq:nfcomp}
\ea
We say that \(w^{[1]}_1\) is the normal form of the first order (subscript) with respect to the linear part \(\mathsf{A}\), that is, at first level (superscript).

Note that the choice of \(\mathcal{C}_i\) will determine the normal form style.

For example, in the
$\Sl$ style, as described in the Introduction, terms are in the first-level normal form if they are in $\ker{\rm ad} (\M).$
In this paper, we introduce the triangular \(\Sl\)-style for feed forward networks.

 To derive the unique normal form, it is essential to first obtain the first-level normal form. Then identify the leading nonlinear terms from this form and define new grades accordingly. This process allows for further reduction through additional transformations applied to the first-level normal form.
\begin{remark}
In this paper, we follow the theoretical normal form approach for singularity and utilize the unique second level normal form. It is noteworthy that we do not have a predefined method for obtaining the unique normal form; however, one can refer to Sanders and Baider's work for guidance, \cite{baider1991unique,baider1992further}. The approach should be tailored to suit the specific problems at hand. We provide a comprehensive explanation, outlining the step-by-step process to determine the unique normal form for both 
\(2\)D and \(3\)D.
\end{remark}
\section{Triangular \(\Sl\)-style}\label{sec:algorithm}

 	This section provides the main theorems on the normal form of a feed forward network with a single node. First, we specify the configuration space for our dynamical system 
	\ba\label{eq:space} 
\mathfrak{r}^i&:=&\mathbb{R}[[x_i,\cdots, x_n]],
	\\\label{eq:space2}
	\mathfrak{t}^n&:=&
	\oplus_{i=1}^{n}\mathfrak{r}^i\otimes \langle \frac{\partial}{\partial x_i}\rangle.
	\ea

In \(\mathfrak{t}^n\), we look at the vector fields with irreducible nilpotent part, which have the form 
	
  \begin{align}\label{eq:eqs}
      \begin{cases} 
\dot{x}_1&={x_{2}}+f_1({x_{1}},{x_{2}},\cdots, x_n),
	\\
\dot{x}_2&={x_{3}}+f_2({x_{2}},\cdots, x_n),
	\\
	&\vdots
	\\
	\dot{x}_{n-1}&= x_{n}+f_{n-1}(x_{n-1},x_n),
	\\
	\dot{x}_n&= f_n(x_n),
	\end{cases}
 \end{align}

	where \(f_i(x_i,..,x_{n})\) is the nonlinear polynomial.
	We shall define the following notations:
	\begin{itemize}

		\item 
		$\mathfrak{\g}^p:$ The \(p\)-th component of 
		$\mathfrak{t}$
		\ba\label{eq:l}
		\mathfrak{\g}^p&=&\langle f_p(x_p,x_{p+1},\cdots,x_n)\frac{\partial}{\partial x_p}\rangle.
		\ea
	
  Note that $ \mathfrak{\g}^p= \mathfrak{r}^p\frac{\partial}{\partial x_p}.$
		\item  Define the filtered vector spaces  $0\subset\mathfrak{t}^1\subset\cdots\subset\mathfrak{t}^{n-1}\subset\mathfrak{t}^n=\mathfrak{t}$ by
		\bas
		\mathfrak{t}^p&=&\oplus_{i=1}^{p}\mathfrak{\g}^{i}.
		\eas
				\item
		$\N:=\sum_{i=1}^{n-1} x_{i+1}\frac{\partial}{\partial x_i}\in \mathfrak{t}^{n-1},
		\,\,\, \mbox{for all } n \ge 2$.

  \item $\mathsf{G}^p:= \mathfrak{t}/{\mathfrak{t}^{p-1}}.$ 
	\end{itemize}
	\begin{proposition}\label{pro:lie}
		The following holds. 
		\bas
		{[\mathfrak{t}^p,\mathfrak{t}^q]} \subseteq \mathfrak{t}^{\min(p,q)}.
		\eas
	\end{proposition}
	\begin{proof}
		Following the  definition of \(\mathfrak{t}^p\) we find
		\bas
		[\oplus_{i=1}^{p} \mathfrak{\g}^i,\oplus_{j=1}^{q} \mathfrak{\g}^j]=\oplus_{i=1}^{p}\oplus_{j=1}^{q}[\mathfrak{\g}^i,\mathfrak{\g}^j],
		\eas
		on the other hand  for given  $f_i(x_i,\cdots,x_n) \frac{\partial}{\partial x_i}\in  \mathfrak{\g}^i $  and  $g_j(x_j,\cdots,x_n) \frac{\partial}{\partial x_j}\in  \mathfrak{\g}^j $  we obtain
		\bas
		[f_i(x_i,\cdots,x_n)\frac{\partial}{\partial x_i}, g_j(x_j,\cdots,x_n)\frac{\partial}{\partial x_j}]&=&
		f_i(x_i,\cdots,x_n) \frac{\partial}{\partial x_i}
		(g_j(x_j,\cdots,x_n))\frac{\partial}{\partial x_j}
  \\&&
		-g_j(x_j,\cdots,x_n)\frac{\partial}{\partial x_j}(f_i(x_i,\cdots,x_n))\frac{\partial}{\partial x_i}
		\\
		&\subseteq&\mathfrak{\g}^{\min(i,j)}.
		\eas
  Hence, 
  $[\mathfrak{t}^p,\mathfrak{t}^q]\subseteq\oplus_{i=1}^{p}\oplus_{j=1}^{q}\mathfrak{\g}^{\min(i,j)}\subseteq\mathfrak{t}^{\min(p,q)}$.
		
	\end{proof}

	\begin{corollary}\label{cor:Lie}
		\bas 
		[\N,{\mathfrak{t}}^{q}] \subseteq {\mathfrak{t}}^{\min(q,n-1)},
		\eas
		therefore,
		\begin{equation}
		\left\{
		\begin{aligned}
		{\rm ad}_{\N}({\mathfrak{t}}^q) & \subseteq{\mathfrak{t}}^{n-1}\,\,\,\, \mbox{if}\,\,\,\, q=n,\\
		{\rm ad}_{\N}({\mathfrak{t}}^q)  & \subseteq {\mathfrak{t}}^q\,\,\,\,\,\,\,\,\,\, \mbox{if}\,\, q<n.
		\end{aligned} \right.
		\end{equation}
	\end{corollary}

	\begin{theorem}\label{thm:Lop}
		${\rm{ad}}_{\N}$ maps $\mathsf{G}^p $ to $\mathsf{G}^p$ as follows: Let \(\sum_{q=p}^n g_q(x_q,\cdots,x_n) \otimes \frac{\partial}{\partial x_q}\) be a transformation generator leaving \(\mathfrak{t}^{p+1}\) invariant. In particular, this requires \(\mathsf{L}_{\N} g_{q}=g_{q+1}\) for \(q=p+1,\cdots,n-1\) and \( g_{n}=0\). This in turn tells us that 
  \(g_q=x_q k_{q+1}+h_{q+1}, k_{q+1}\in \ker \mathsf{L}_{\N}\), and \(g_{q+1}=x_{q+1} k_{q+1}+\mathsf{L}_{\N} h_{q+1}\).
  Therefore, \(g_q\) is built from the monomials \(x_q^{m_1} x_{q+1}^{m_2} \cdots x_n^{m_{n-q+1}}\), with \(m_i\leq i\) for \(i<n-q+1\).

  Then
		\bas 
  {\rm{ad}}_{\N}(\sum_{q=p}^n g_q(x_q,\cdots,x_n) \otimes \frac{\partial}{\partial x_q})=(\mathsf{L}_{\N}\left(g_p(x_p,\cdots,x_n)\right)
  -g_{p+1}(x_{p+1},\cdots,x_n) )\otimes \frac{\partial}{\partial x_p}\mod{\mathfrak{t}^{p-1}},
  \eas
		where $\mathsf{L}_{\N}$ is an operator that  act on the  $\mathfrak{r}^p$ by 
		\bas 
		\mathsf{L}_{\N}&:& {\mathfrak{r}^p}  \rightarrow  {\mathfrak{r}^p}  
		\\
		\mathsf{L}_{\N}(g_p(x_p,x_{p+1},\cdots,x_n))&=&\sum_{i=p}^{n}  x_{i+1}
		\frac{\partial}{\partial x_i}g_p(x_p,x_{p+1},\cdots,x_n).
		\eas
	\end{theorem}
	\begin{proof}
		The proof follows from the argument
		
		\bas
		{\rm ad}_{\N}\sum_{q=p}^n g_q(x_q,x_{q+1},\cdots,x_n)
		\frac{\partial}{\partial x_q}&=&
		\sum_{i=p}^{n-1}  x_{i+1}
		\frac{\partial}{\partial x_i}(g_p(x_p,x_{p+1},\cdots,x_n))
		\frac{\partial}{\partial x_p}
		- g_p(x_p,x_{p+1},\cdots,x_n)\frac{\partial}{\partial x_{p-1}}
		\\&&- g_{p+1}(x_{p+1},\cdots,x_n)\frac{\partial}{\partial x_{p}}
  \\
		&=&
		\mathsf{L}_{\N}\left(g_p(x_p,x_{p+1},\cdots,x_n) \right)- g_{p+1}(x_{p+1},\cdots,x_n)\frac{\partial}{\partial x_p}
		\,\,\,\,\, \mbox{mod}\,\, \mathfrak{t}^{p-1}.
		\eas
	\end{proof}
	With the Lie algebraic structure established and a thorough examination of the Lie product within the feed forward network completed, we are now prepared to define the triangular $\Sl$-style normal form. This algorithm works deductively as follows: For a given network with dimension $n,$ construct $\Sl$ near the nilpotent linear part $\N$ using the Jacobson-Morosov theorem \cite{knapp1996lie}. Before introducing the algorithm, we present the constant normal form.

	\begin{lemma}[Constant normal form]\label{lem:Cons}
		Here we have the constant vector:
		\bas 
		\sum_{i=1}^n \nu_i \frac{\partial }{\partial x_{i}},
		\eas
		where \( \nu_i\) are constants. Now, we can write it as \(\nu_j \frac{\partial }{\partial x_{j}}=[ \N, -\nu_j \frac{\partial }{\partial x_{j+1}}],\)  for all \(1\leq j \leq n-1.\) Therefore, only the constant term in form 
		\(\nu_n\frac{\partial }{\partial x_{n}}\) remain in the zero degree terms  of normal form {\em(}complement{\em )}.
	\end{lemma} 
 
	\begin{theorem}[An algorithm: triangular \(\Sl\)-style normal form]\label{thm:sl2triple}	
		First we  define the following \(\Sl\)-triples \(\langle 	{\N}_p,	{\H}_p,	{\M}_p \rangle\)
		for all \(1\leq p \leq n\)
		\ba
		\nonumber
		{\N}_p&:=&\sum_{i=p}^{n-1} x_{i+1} \frac{\partial }{\partial x_i},
  \\\nonumber
		{\H}_p&:=& \sum_{i=p}^{n} (n+p-2i) x_{i} \frac{\partial }{\partial x_{i}},
		\\\label{eq:inductivesl2}
		{\M}_p&:=& \sum_{i=p+1}^{n} (i-p) (n-i+1)  x_{i-1} \frac{\partial }{\partial x_{i}},
		\ea
		with these relations between the operators:
		\bas
		[	{\M}_p,	{\N}_p]=	{\H}_p,\,\,\,\, [	{\H}_p,{\N}_p]=-2 {\N}_p,\,\,\,\, [	{\H}_p,{\M}_p]=2 {\M}_p.
		\eas
		Observe that \({\N}_1={\N},\) see also {\em Figure} \eqref{fig:inductive}.
		In the  space of triangular functions \(\mathfrak{r}^p\), the following holds:
		\bas 
		\mathfrak{r}^p= \mathsf{L}_{{\N}_p}(\mathfrak{r}^p)\oplus  {\ker}(\mathsf{L}_{{\M}_p})\mid 	\mathfrak{r}^p,
		\eas
   where
   \bas
\mathsf{L}_{\N}(g_p(x_p,x_{p+1},\cdots,x_n))&=&\sum_{i=p}^{n}  x_{i+1}
		\frac{\partial}{\partial x_i}g_p(x_p,x_{p+1},\cdots,x_n),
  \\
\mathsf{L}_{{\M}_p}(f_p(x_p,x_{p+1},\cdots,x_n))&=& \sum_{i=p+1}^{n} (i-p) (n-i+1)  x_{i-1} \frac{\partial }{\partial x_{i}}f_p(x_p,x_{p+1},\cdots,x_n).
     \eas
		Now, we define  the normal form on \(\mathfrak{r}^p\) 
		\bas 
		\mathcal{C}_p:=\langle {\bar{f}}_p(x_p,\cdots,x_n) \frac{\partial}{\partial x_{p}}
		\mid  {\bar{f}}_p \in  {\ker}(\mathsf{L}_{{\M}_p})\mid 	\mathfrak{r}^p\rangle,\,\,\,\, \mbox{for all }\,\,\, 1\leq p \leq n.
		\eas
  Note that for \(p=n\)  the \(\Sl\)-triple is trivial hence the complement is 
  \(\mathfrak{r}^n=
  \mathbb{R}[[x_n]].\) The normal form  of \eqref{eq:eqs} is  
  \begin{align}\label{eq:eqsNF}
      \begin{cases} 
\dot{x}_1&=
  {x_{2}}+{\tilde f}_1(x_1,\cdots,x_n),\,\, \,\,\,\,\qquad{\tilde f}_1\in {\ker}(\mathsf{L}_{{\M}_1})\mid 	\mathfrak{r}^1,
		\\
		\dot{x}_2&=x_{3}+
		{\tilde f}_{2}(x_2,\cdots,x_n),\,\,\,\, \,\,\,\,\qquad
		{\tilde f}_{2},\in  {\ker}(\mathsf{L}_{{\M}_2})\mid 	\mathfrak{r}^2,
		\\
		&\vdots
		\\
		\dot{x}_n&=	f_n(x_n),	\quad\quad\qquad\qquad\,\, \,\,\,\,\qquad f_n\in  {\ker}(\mathsf{L}_{{\M}_n})\mid 
  \mathfrak{r}^n.
	\end{cases}
 \end{align}
  	\end{theorem}
  \begin{remark}
    Note that no terms would be removed from the last component,
since \({\M}_n=0.\) Therefore, the normal form in the  \(n\)-th component  is the same as the original \(n\)-th  component, i.e. \(f_n(x_n).\)
  \end{remark}
   
		\newlength{\tempwidth}
		\newlength{\tempheight}
		
		\newcommand{\Mark}[2]
		{\tikz[remember picture, overlay]%
			{\node[anchor=base](#1){$#2$};}}
\newpage
		\begin{figure}
		    \centering
		  \begin{equation}\nonumber
		{\N}=\begin{array}{c}
		\textrm{${\N}={\N}_{1}$}\\
		\overbrace{%
			\begin{pmatrix}
			0& &1&0&0&\cdots&0&0& \\
			\vphantom{\overbrace{\strut}} \\
			0& &\Mark{NW}{0}&1&0&\cdots&0&0\\
			0&&0&0&1&\vdots&\vdots&\vdots\\
			0&&0&0&0&\vdots&0&0\\
			\vdots&&\vdots&\vdots&\vdots&\vdots&1&0\\
			\vdots&&\vdots&\vdots&\vdots&\vdots&0&1\\
			0&&0&0&0&\cdots&0&\Mark{SE}{0}\\
			\vphantom{\underbrace{\strut}}
			\end{pmatrix}}
		\end{array}
		\end{equation}
		    \caption{Nested \(\N_p,\) for all  $1\leq p \leq n $.}\label{fig:inductive}
		\end{figure}
		\begin{tikzpicture}[remember picture,overlay]
		\pgfextracty{\tempheight}{\pgfpointdiff{\pgfpointanchor{SE}{south}}%
			{\pgfpointanchor{NW}{north}}}
		\global\tempheight=\tempheight
		\pgfextractx{\tempwidth}{\pgfpointdiff{\pgfpointanchor{NW}{west}}%
			{\pgfpointanchor{SE}{east}}}
		\global\tempwidth=\tempwidth
		\path (NW.north west) +(0.5\tempwidth,-0.5\tempheight) node{$\displaystyle
			\begin{array}{c}
			\textrm{${\N}_{2}$}\\
			\overbrace{\underbrace{%
					\begin{pmatrix}
					\rule{0pt}{\tempheight}\rule{\tempwidth}{0pt}
					\end{pmatrix}}}\\
			\end{array}$};
		\end{tikzpicture}

	\section{Classification of quadratic normal forms via invariant theory}\label{sec:inv}
In this section, we provide the quadratic normal form for the given feed forward network \eqref{eq:eqs} for arbitrary dimension \(n.\) To achieve this, we classify the generators of $ {\ker}(\mathsf{L}_{{\M}_p})$ by looking for the transvectant of \(x_p\)
(since \(x_p\) is the trivial generator of   $ {\ker}(\mathsf{L}_{{\M}_p})$)
 with itself.

	\begin{lemma}\label{lem:trans}
		For given \(\Sl\)-triple  \(\langle 	{\N}_p,	{\H}_p,	{\M}_p \rangle \) in \eqref{eq:inductivesl2} the  quadratic  kernel of 
		\(	{\M}_p\) or equivalently  the  quadratic generators of  $ {\ker}(\mathsf{L}_{{\M}_p})\mid 	\mathfrak{r}^p$ is
		\ba\label{eq:inv2}
		\tau^k_p= \sum_{j=0}^{k} (-1)^{k-j} 
 \binom{n-p-(k-j}{j}\binom{n-p-j}{k-j}x_{p+k-j}x_{p+j},
  \ea
		for all  $1 \leq p\leq n$ and  $0\leq  k\leq n-p$.
	\end{lemma}
	
 \begin{proof}
    The transvectant of $\mathsf{f}$ and $\mathsf{g}$ is given by:
    \begin{align*}
    {(\mathsf{f},\mathsf{g})}^{(k)} &= \sum_{i+j=k} (-1)^i \binom{\omega_{\mathsf{f}}-i} {k-i}
    \binom{\omega_\mathsf{g}-j}{k-j} \mathsf{f}^{(i)}\mathsf{g}^{(j)},
    \end{align*}
    where $\mathsf{f}^{(i)}:=\mathsf{L}_{\N_p}^i(\mathsf{f})$ (the $i$th action of $\mathsf{L}_{{\N}_p}$ on $\mathsf{f}$) and $\omega_\mathsf{f}$ and $\omega_g$ are the weights (eigenvalues) of $\mathsf{f}$ and $\mathsf{g}$, respectively, that is, $ \mathsf{L}_{{\H}_p}\mathsf{f}=\omega_\mathsf{f} \mathsf{f}$ and $ \mathsf{L}_{{\H}_p}\mathsf{g}=\omega_\mathsf{g} \mathsf{g}$, see \cite[12.3.2]{SVM2007}. As can be seen from \eqref{eq:inductivesl2}, $x_p\in\ker(\mathsf{L}_{{\M}_p})$. 
    Hence, $x_p^{(i)}=x_{p+i}$ and $\omega_{x_p}=n-p$. Computing the transvectant of this kernel with itself via the above formula generates the other quadratic kernels in addition to the obvious \(x_p^2\). Observe that
    \begin{align*}
    \tau_p^k&:=(x_p,x_p)^{(k)} = \sum_{j=0}^{k} (-1)^{k-j} 
    \binom{\omega_{x_p}-(k-j)}{j}
    \binom{\omega_{x_p}-j}{k-j}x_p^{(k-j)}x_p^{(j)} \\
    &= \sum_{j=0}^{k} (-1)^{k-j} 
    \binom{n-p-(k-j)}{j}
    \binom{n-p-j}{k-j} x_{p+k-j}x_{p+j}, \quad p+k\leq n.
    \end{align*}
    We have obtained the desired result. Note that for odd values of $k$, the transvectant is anti-symmetric; hence, $\tau_p^k=(x_1,x_1)^{(2l+1)}=0$ for $l=0,\cdots,\infty$. We remark here that \(\tau_p^{2l}\simeq x_{l+p}^{2}\) modulo the image of \(\mathsf{L}_\mathsf{N}\).
    This equivalence could be used to reformulate the next Theorem.
\end{proof}

	 \begin{theorem}[Quadratic normal form]\label{thm:ndnf}
  We can achieve a quadratic inner formal normal form for Equation \eqref{eq:eqs} through a sequence of invertible transformations. This normal form is given by the following expression:
		\bas 
		v^{[1]}:=\N+
		\sum_{1\leq p+k\leq n} a^k_p\tau^k_p \frac{\partial }{\partial x_p},
		\eas
		where the $a^k_p$ are normal form coefficients, determined by the coefficients of the quadratic terms of the original system.
	\end{theorem}
 \begin{proof}
     The proof that this is a normal form follows from Theorem \ref{thm:sl2triple} and  Lemma \ref{lem:trans}.
     In Section \ref{SEC:NFAPP} we show that it satisfies the Cushman-Sanders test, and therefore describes the complete normal form in grade \(1\) for all \(n\).
 \end{proof}

	 \subsection{Normal form approval}\label{SEC:NFAPP}
In this section, we utilize a generating function method to verify the correctness of the quadratic normal forms given by  Theorem \ref{thm:ndnf}. How does it assist us? It allows for the validation of both the number and types of terms present in the normal form. Initially, one must identify the terms of the normal form. Subsequently, by utilizing the generating function associated with the normal form, one can perform the validation.

To use this test, we associate an element \(\mathsf{a}\) in normal form, \(\ker(\mathsf{L}_{{{\M}}_p})\), with \(u^\omega t^d\) in which \(d\) is the degree and \(\omega\) is the weight of \(\mathsf{a}\). Then we sum up all these elements, which should be equal to \(\mathcal{P}^{n,p}_2(d,u)\), given in \eqref{cor:app}. It turns out that the normal form is correct. See Section \ref{subsec:examples} for examples.

This method for determining the normal form was originally introduced by Cushman and Sanders in \cite{cushman56nilpotent}. For a more comprehensive understanding of this method, we refer to this source and \cite{sylvester1879tables,SVM2007}.

	\begin{theorem}\label{thm:gen}
		The following holds:
			The generating function for quadratic terms of 
			\({\ker}(\mathsf{L}_{{\M}_p})\) given  in \eqref{eq:inductivesl2} is
			\ba\label{eq:genern}
			\mathcal{P}_2(d,u)=\frac{t^2}{(1-u^2 d)(1-d^2)}.
			\ea
	
  \end{theorem}
	\begin{proof}
		The idea of proving this theorem is applying the Hermite reciprocity \cite[Theorem 6.31]{olver1987invariant} which is 
		\bas
		S^mS^2\mathbb{C}^2\simeq S^2S^m\mathbb{C}^2,
		\eas
		where \(S^2 \mathbb{C}^2\) is  the space of polynomial of degree $2$ in two variables, therefore \(S^m S^2 \mathbb{C}^2\) is the polynomial space of degree $m$ with three variables. And 
		where \(S^m \mathbb{C}^2\) is  the space of polynomial of degree $m$ in two variables, hence   \(S^2 S^m \mathbb{C}^2\) is the space of polynomial of degree 
		$2$ with $m+1$ variables. 
		
		The generating function for $S^mS^2\mathbb{C}^2$ is 
		$$
		\mathcal{P}_2(t,u)=\frac{d^2}{(1-u^2 t)(1-t^2)},$$ see 
		\cite[Page 297]{SVM2007}. We note that $d^2$ (which expresses  the 
		\(3\)D space here, since a quadratic expression in two variables has three coefficients) is not written in this reference. However,  we put it here to make the application of Hermite reciprocity clear.  Now, by employing  Hermite reciprocity;
		exchanging  $t$ and $d$ in $\mathcal{P}_2(t,u)$
		we  obtain
		\ba\label{eq:gen3n}
		\mathcal{P}_2{(d,u)}=\frac{t^2}{(1-u^2 d)(1-d^2)}.
		\ea
		If we apply the Cushman-Sanders test, we obtain
  	\ba\label{eq:gen3ntest}
		\mathcal{P}_2{(d)}=\frac{\partial}{\partial u} \frac{ut^2}{(1-u^2 d)(1-d^2)}|_{u=1}=\frac{t^2}{(1-d)^3},
		\ea
  which brings out the generating function for quadratic polynomials with \(m\) variables, $S^2S^m\mathbb{C}^2$.  
		
	\end{proof}
	\begin{remark}
    The application of Hermite reciprocity is not restricted to triangular systems but can also be used to generate normal form theorems to a limited degree for arbitrary dimensions, just as the general normal form has been determined for limited dimensions to an arbitrary degree.
	\end{remark}
     

	\begin{corollary}[Normal form approval]\label{cor:app}  The  quadratic  generating function of \({\ker}(\mathsf{L}_{{\M}_p})\) for fixed \(n\) is
		\bas 
		\mathcal{P}^{n,p}_2(u)=t^2 u^{2\alpha}\sum_{l=0}^{{[\frac{n-p}{2}]}}u^{4l},
		\eas

for odd values of \(p,\)
\(\alpha\) equals $0$, and for even values of   \(p,\) 
\(\alpha\)  equals $1$.		
		\end{corollary}
  \begin{proof}
     By taking the  Taylor expansion of \eqref{eq:genern} respect to \(d\) one finds
		\ba\label{eq:gpd}
		\nonumber
		\mathcal{P}_2(d,u)&\approx& {t}^{2}+{t}^{2}{u}^{2}d+ 
  {t}^{2}\left( {u}^{4}+1 \right) {d}^{2
		}+ {t}^{2}\left( {u}^{6}+{u}^{2} \right) {d}^{3}+{t}^{2} 
  \left( {u}^{8}+{u}^{4}+1 \right) {d}^{4}
		\\&&+
		{t}^{2}\left({u}^{10
		}+{u}^{6}+{u}^{2} \right) {d}^{5}
		+ {t}^{2}\left( {u}^{12
		}+{u}^{8}+{u}^{4}+1 \right) {d}^{6}+O \left( {d}^{
			7} \right),
		\ea
		the coefficient of $d^{p-1}$ indicates the generating function for \(\ker(\mathsf{L}_{{\M}_p}).\) 
  See \S\ref{subsec:examples} for the interpretation of these terms in terms of concrete functions in the normal form.
  A given term in the expansion is of the form \(t^2(u^2d)^{k}d^{2l}\), representing \(\tau^{2l}_p=(x_p,x_p)^{(2l)}\) with eigenvalue \(2k\) in dimension \(k+2l+2\).
  \end{proof}
	
	\section{General computations and examples}\label{sec:3d}	
 Let us now illustrate the approach outlined in Section 2 with examples.
\subsection{$3$D normal form}
In this section, we intend to calculate all transformations and find the quadratic normal form for the feed forward network within three dimensions. We apply the triangular \(\Sl\)-style given in Section \ref{sec:algorithm}.

Consider the following three-dimensional nilpotent  feed forward system:
\begin{equation} \label{eq:3Dquadratic}
v^{[0]}= \left\{ 
    \begin{array}{l} 
        \dot{x_{1}} = 
        {{x_{2}}} + \epsilon a_{1} {x_{1}} + \epsilon a_{2} {{x_{2}}} + \epsilon a_{3} {{x_{3}}} + a_{4} {{x_{2}}} {{x_{3}}} + a_{5} x_2^{2} + a_{6} x_3^{2} + a_{7} {x_{1}} {{x_{3}}} + a_{8} x_1^{2} + a_{9} {x_{1}} {{x_{2}}}, \\ 
        \dot{x_{2}} = {{x_{3}}} + \epsilon b_{2} {{x_{3}}} + \epsilon b_{1} {{x_{2}}} + b_{3} {{x_{2}}} {{x_{3}}} + b_{4} x_2^{2} + b_{5} x_3^{2}, \\
        \dot{x_{3}} = \epsilon c_{1} {{x_{3}}} + c_{2} x_3^{2},
    \end{array} 
\right.
\end{equation}

	or equivalently,
\bas
		\begin{pmatrix}
			{\dot x}_1
			\\
			{\dot x}_2
			\\
			{\dot x}_3
		\end{pmatrix}&=&
		\begin{pmatrix}
			\epsilon  a_{1} & 1+\epsilon a_2 & \epsilon  a_{3} 
			\\
			0 & \epsilon  b_{1} & 1+\epsilon b_2
			\\
			0 & 0& \epsilon  c_{1} 
		\end{pmatrix}
		\begin{pmatrix}
			x_1
			\\
			x_2
			\\
			x_3
		\end{pmatrix}+\begin{pmatrix}
			a_{4} {{x_{2}}} {{x_{3}}}+a_{5} x_2^{2}+a_{6} x_3^{2}+a_{7} {x_{1}} {{x_{3}}}+a_{8} x_1^{2}+a_{9} {x_{1}} {{x_{2}}}
			\\
		b_{3} {{x_{2}}} {{x_{3}}}+b_{4} x_2^{2}+b_{5} x_3^{2} 
			\\
		c_{2} x_3^{2}
		\end{pmatrix},
		\eas
	in which \(a_i,b_i\) and \(c_i\)  are given  real coefficients  and \(\epsilon\) is  a small  parameter.  To transform the vector field to normal form, we are taking the following steps:
	\begin{itemize}
		\item (Constant normal form). The parameter \(\nu \frac{\partial }{\partial {x_{3}}}\) represents the zero-order normal form; see Lemma \eqref{lem:Cons}. Note that the zero-level transformations do not change the architecture of the original system. 
		\item (Linear normal form).
		Define, 
		\bas 
		\mathsf{A}(\epsilon) =\begin{pmatrix}
			\epsilon  a_{1} & 1+\epsilon a_2 & \epsilon  a_{3} 
			\\
			0 & \epsilon  b_{1} & 1+\epsilon b_2
			\\
			0 & 0 & \epsilon  c_{1} 
		\end{pmatrix},
		\eas
		where \(\mathsf{A}(0)=\mathsf{A}.\) Employing  this near identity transformation
		\bas 
		\mathsf{T}(\epsilon)=\begin{pmatrix}
			1 & \epsilon  t_{1} & 0 
			\\
			0 & 1 & 0 
			\\
			0 & 0 & 1 
		\end{pmatrix},
		\eas
  with \(t_1 \in \mathbb{R}\)
		 into the \(\mathsf{A}(\epsilon)\) we find 
		\bas 
		\mathsf{T}^{-1}(\epsilon) \mathsf{A}(\epsilon) \mathsf{T}(\epsilon)=
		\begin{pmatrix}
			\epsilon  a_{1} & \epsilon^{2} a_{1} t_{1}-\epsilon^{2} b_{1} t_{1}+\epsilon  a_{2}+1 & \epsilon  a_{3}-\epsilon  t_{1} \left(1+\epsilon  b_{2}\right) 
			\\
			0 & \epsilon  b_{1} & 1+\epsilon  b_{2} 
			\\
			0 & 0 & \epsilon  c_{1} 
		\end{pmatrix}.
		\eas
		By solving $( a_3 - t_1(\epsilon b_2 + 1) ) \epsilon$ respect to $t_1$ we find $t_1=\frac{a_{3}}{1+\epsilon  b_{2}}.$ Therefore, 
		\ba\label{eq:eq1}
		\mathsf{T}^{-1}(\epsilon) \mathsf{A}(\epsilon) \mathsf{T}(\epsilon)=\bar{\mathsf{A}}(\epsilon)=\begin{pmatrix}
			\epsilon  a_{1} & 1+
			\epsilon \delta
			& 0 
			\\
			0 & \epsilon  b_{1} & 1+\epsilon  b_{2}
			\\
			0 & 0 & \epsilon  c_{1} 
		\end{pmatrix},
		\ea
		where \[\delta=a_2+
  \frac{\epsilon a_3\left(a_1-b_1\right)}{1+\epsilon\,b_2}
.\]  
For a study of versal deformations of non semisimple linear systems, refer to \cite{mokhtari2019versal}.
 As expected, the normal form is generated by those elements (or terms) that belong to the
  \( {\ker}(\mathsf{L}_{{\M}_p}), 
  \mbox{for all}\, p=1,2,3,\) we have
  
  \[	{\ker}(\mathsf{L}_{{\M}_1})=0,\,\,\,	{\ker}(\mathsf{L}_{{\M}_2})={x_{2}} \frac{\partial }{\partial {x_{3}}},
  \mbox{and}\,\,\, {\ker}(\mathsf{L}_{{\M}_3})={x_{1}}\frac{\partial }{\partial {x_{2}}}+2{x_{2}} \frac{\partial }{\partial {x_{3}}}.\]
  Hence the linear generators are \({x_{3}} \frac{\partial }{\partial {x_{3}}}, {x_{2}} \frac{\partial }{\partial {x_{2}}},
  \) and \(
  {x_{1}} \frac{\partial }{\partial {x_{1}}}.\)
Now, the linear part is in versal deformation, by applying the transformation \(\mathsf{T}(\epsilon).\) Before simplifying the quadratic terms we need to apply the linear transformation \(\mathsf{T}(\epsilon)\) to the nonlinear part of \eqref{eq:3Dquadratic},
 
	\ba\label{eq:L1}
		\begin{pmatrix}
			{x_{1}}
			\\
			{x_{2}}
			\\
			{x_{3}}
		\end{pmatrix}
		\mapsto\begin{pmatrix}
			{x_{1}} +\frac{\epsilon  a_{3} {x_{2}}}{1+\epsilon  b_{2}} 
			\\
			{x_{2}}
			\\
			 {x_{3}}
		\end{pmatrix},
		\ea
  or equivalently we apply this transformation:
  \bas
  Y_1=\frac{\epsilon a_3 {x_{2}}}{1+\epsilon b_2} \frac{\partial }{\partial {x_{1}}}.
  \eas
  We find

		\ba\label{eq:3DFL}
  \nonumber
		v^{[1]}=v^{[0]}+[Y_1,v^{[0]}]&=&\left((1+\epsilon \delta) {x_{2}}+\epsilon a_1{x_{1}}
		+a^{(1)}_{4}  {x_{2}} {x_{3}}+a^{(1)}_{5} x_2^2+
		a^{(1)}_{6} x_3^{2}+a_{7} {x_{1}} {{x_{3}}}+a_{8} x_1^{2}+a^{(1)}_{9} {x_{1}} {{x_{2}}}\right)\frac{\partial }{\partial {x_{1}}}
  \nonumber
		\\&&
		+\left((1+\epsilon b_2){x_{3}}+\epsilon b_1 {x_{2}}+b_3{x_{2}} {x_{3}}+b_4 x_2^{2}+b_5 x_3^{2}\right) \frac{\partial }{\partial {x_{2}}}
  \nonumber
  \\&&
  +
		\left(\nu+\epsilon c_1 {x_{3}}+c_2x_3^{2}\right)\frac{\partial}{\partial {x_{3}}},
		\ea
 where  the coefficients of the original system are updated to the following
		\ba\label{eq:coeffnf1}
  \nonumber
		a^{(1)}_{4} &=& a_4+
 \frac{\epsilon a_3\left(a_7-b_3\right)}{1+\epsilon\,b_2},
		\\\nonumber
		a^{(1)}_{5}&=&a_5 +\frac{\epsilon a_3\left(a_9-b_4\right)}{1+\epsilon\,b_2}+{\frac {{
\epsilon}^{2}a_3^{2}a_8}{ \left(1+\epsilon\,b_2
 \right) ^{2}}},
 \\\nonumber
a^{(1)}_{6}&=& a_6-{\frac {\epsilon b_5\,a_3}{1+\epsilon\,b_2}},
		\\
		a^{(1)}_{9}&=&
	a_9+{\frac {2\epsilon\,a_3a_8}{1+\epsilon\,b_2}}.
		\ea
	
\item
   (Quadratic normal form of the second component).
		In the next step, we intend to find the quadratic normal form:
as our algorithm suggests, we eliminate the following terms from the second equation which are $\langle {x_{2}}{x_{3}}, x_2^{2}\rangle\frac{\partial}{\partial {x_{2}}}$ by the following transformation. 
		
		\bas
		Y_2&:=&\left(\frac{b_{5} {x_{2}} {x_{3}}}{1+\epsilon  b_{2}}+\frac{\left(b_{3}+\epsilon  b_{2} b_{3}-\epsilon b_{5} c_{1}\right) x_2^{2}}{2 \left(1+\epsilon  b_{2}\right)^{2}}\right) \frac{\partial}{\partial {x_{2}}}.
		\eas
		Hence, 
		\bas 
		{v}^{[2]}&=&{\exp\rm ad}(Y_2)(v^{[1]})=v^{[1]}+[Y_2,v^{[1]}]+\cdots
  \\
  &=&
  \left(\epsilon a_1{x_{1}}+(1+\epsilon \delta) {x_{2}}
		+a^{(2)}_{4}  {x_{2}} {x_{3}}+a^{(2)}_{5} x_2^2+
		a^{(1)}_{6} x_3^{2}+a_{7} {x_{1}} {{x_{3}}}+a_{8} x_1^{2}+a^{(1)}_{9} {x_{1}} {{x_{2}}}\right)\frac{\partial }{\partial {x_{1}}}
		\\&&
		+\left(\epsilon b_1 {x_{2}}+(1+\epsilon b_2){x_{3}}+b^{(2)}_4 x_2^{2}\right) \frac{\partial }{\partial {x_{2}}}+
		\left(\nu+\epsilon c_1 {x_{3}}+c_2x_3^{2}\right) \frac{\partial }{\partial {x_{3}}}+\cdots,
		\eas

  with the following coefficients:
  \ba
  \nonumber
  a_4^{(2)}&=&a_4+{\frac {\delta\,\epsilon\,b_5}{1+\epsilon\,b_2}}+{
\frac {b_5}{1+\epsilon\,b_2}},
\\\nonumber
  a_5^{(2)}&=&
 {\frac { \left( \delta\,b_2b_3-\delta\,b_5c_{{1}}+2\,
a^{(1)}_5{b_2}^{2} \right) {\epsilon}^{2}+ \left( \delta\,b_3-b
_5c_{{1}}+ \left( 4\,a^{(1)}_5+b_3 \right) b_2 \right) 
\epsilon+2\,a^{(1)}_5+b_3}{ 2\left( 1+\epsilon\,b_2 \right) ^{2}
}},
\\\label{eq:b2}
  b^{(2)}_4&=&
  -{\frac {\left({\epsilon}^{2} b_2b_{{1}}b_3
-{\epsilon}^{2}b_{{1}}b_5c_{{1}}-2{\epsilon}^{2}\,{b_2}^{2}b_{
{4}}+\epsilon\,b_{{1}}b_3-4\,\epsilon\,b_2b_4-2\,b_4
 \right) }{2 \left( 1+\epsilon\,b_2 \right) ^{2}}}.
  \ea
\item
 (Quadratic normal form of the first component).
		In the last step, we put the first component in the normal form, which is the elimination of these terms:
		$\langle {x_{2}} {x_{3}}, x_1^{2}, {x_{1}} {x_{2}} 
		\rangle
		\frac{\partial }{\partial {x_{1}}}$
		by this transformation:
		\bas 
		Y_3=\left( \alpha_{{1}}x_2x_3+\alpha_2x_2^{2}+\alpha_3x_1^{2
}+\alpha_4x_{{1}}x_2 \right)
\frac{\partial }{\partial {x_{1}}},
		\eas
		with the following coefficients:
		\bas 
		\alpha_1&=&\frac{a^{(1)}_6}{1+\epsilon\,b_{2}},
		\\
		\alpha_2&=&-\frac{1}{2}\,
  \frac{-\epsilon\,a_1 a^{(1)}_6+\epsilon\,a^{(1)}_6 b_{{1}}+\epsilon\,a^{(1)}_6 c_1-\epsilon\,a^{(2)}_4 b_2-a^{(2)}_4}
{ {\left( 1+\epsilon\,b_2 \right)}^{2}},
		\\
        \alpha_{3}&=&-\frac{1}{2}\frac{\epsilon\,b_{1}\alpha_{4}-a^{(1)}_9}{1+\,\epsilon\,\delta},
  \\
  \alpha_4 &=&
  \frac {{\epsilon}^{2}a_1^{2}a^{(1)}_6-3\,{\epsilon}^{2}a_{{1}}a^{(1)}_6b_{{1}}-{\epsilon}^{2}a_{{1}}a^{(1)}_6c_{{1}}+{\epsilon}^{2}a_{{1
}}a^{(2)}_4b_2+{\epsilon}^{2}a_7{b_2}^{2}+2\,{\epsilon}^{2}
a^{(1)}_6{b_{{1}}}^{2}+2\,{\epsilon}^{2}a^{(1)}_6b_{{1}}c_{{1}}-2\,{
\epsilon}^{2}a^{(2)}_4b_{{1}}b_2}{
 {\left( 1+\epsilon\,b_2 \right)}^{2} 
 \left(3+\epsilon\delta+\epsilon\,b_2 \right)}
\\
&&
+\frac{a_7+2\,a^{(2)}_5+2\,{\epsilon}^{2}a^{(2)}_5{b_2}
^{2}+\epsilon\,a_{{1}}a^{(2)}_4+2\,\epsilon\,a_7b_2-2\,\epsilon
\,a^{(2)}_4b_{{1}}+4\,\epsilon\,a^{(2)}_5b_2}{
 \left( 1+\epsilon\,b_2 \right)^{2}\left(3+\epsilon\delta+\epsilon\,b_2\right ) }.
		\eas
	Note that, for removing these three terms the first three transformations are enough. We use the last one to generate \((2{x_{1}}{x_{3}}-x_2^2)\frac{\partial }{\partial {x_{1}}}\) in normal form. 
		The third level normal form is given by:
		\bas 
{ v}^{[3]}&=&{\rm exp}({\rm ad})_{\rm Y_3}(v^{[2]})=v^{[2]}+[Y_3,v^{[2]}]
\\
&=&
\left(\epsilon a_1{x_{1}}+(1+\epsilon \delta) {x_{2}}
		+a^{(3)}_{5} x_2^2+
		a^{(3)}_{7} (x_2^{2}-2{x_{1}} {{x_{3}}})\right)\frac{\partial }{\partial {x_{1}}}
		\\&&
		+\left(\epsilon b_1 {x_{2}}+(1+\epsilon b_2){x_{3}}+b^{(2)}_4 x_2^{2}\right) \frac{\partial }{\partial {x_{2}}}+
		\left(\nu+\epsilon c_1 {x_{3}}+c_2x_3^{2}\right) \frac{\partial }{\partial {x_{3}}},
\eas
with 
\ba
\nonumber
a^{(3)}_{5}&=&\frac{1}{2}{\frac {  {\epsilon}^{2}a_{{1}}b_{{1}}\alpha_4+2\,\epsilon\,
\delta\,a_8-\epsilon\,a_{{1}}a^{(1)}_9+2\,a_8  }{1+\epsilon\,\delta}},
\\\label{eq:nfcoeff}
a^{(3)}_{7}&=&-\frac{1}{2}
\left(\epsilon\,b_2\alpha_4+\alpha_4-a_7\right),
\ea
where  \(a^{(1)}_9 \) is given by \eqref{eq:coeffnf1}. 
	\end{itemize}
Finally, one can use the following scaling to  translate  off-diagonal entries of  \(\bar{\mathsf{A}}(\epsilon)\) to \(1.\)
\bas
x &\mapsto&(1+\epsilon \delta) (1+\epsilon b_2) x,
\\
y&\mapsto &(1+\epsilon b_2) y.
\eas

	\begin{theorem}
 The quadratic  triangular \(\Sl\)-style normal form of 

\bas
		v^{[0]}=	\left\{ \begin{array}{l} \dot {x}_{1} = 
		{{x_{2}}}+
	\epsilon a_{1} {x_{1}}+\epsilon a_{2} {{x_{2}}}+\epsilon a_{3} {{x_{3}}}+a_{4} {{x_{2}}} {{x_{3}}}+a_{5} x_2^{2}+a_{6} x_3^{2}+a_{7} {x_{1}} {{x_{3}}}+a_{8} x_1^{2}+a_{9} {x_{1}} {{x_{2}}},\\ 
				\dot {x}_{2} ={{x_{3}}}+\epsilon b_{2} {{x_{3}}}+\epsilon b_{1} {{x_{2}}}+b_{3} {{x_{2}}} {{x_{3}}}+b_{4} x_2^{2}+b_{5} x_3^{2},\\
   	\dot {x}_{3} = \epsilon c_{1} {{x_{3}}}+c_{2} x_3^{2},
			\end{array} \right.
   \eas
		
  is 
  \bas
		v^{[3]}=	\left\{ \begin{array}{l} \dot {x}_{1} = 
		\epsilon a_1{x_{1}}+(1+\epsilon \delta) {x_{2}}
		+a^{(3)}_{5} x_2^2-
		a^{(3)}_{7} (2{x_{1}}{x_{3}}-x_2^2),\\ 
				\dot {x}_{2} =\epsilon b_1 {x_{2}}+(1+\epsilon b_2){x_{3}}+b^{(2)}_4 x_2^{2},\\
   	\dot {x}_{3} =\nu+\epsilon c_1 {x_{3}}+c_2x_3^{2},
			\end{array} \right.
   \eas
  or equivalently 
  \bas
  \begin{pmatrix}
			{\dot x}_1
			\\
			{\dot x}_2
			\\
			{\dot x}_3
		\end{pmatrix}=
  \begin{pmatrix}
			0
			\\
			0
			\\
			\nu
		\end{pmatrix}+
		\begin{pmatrix}
			\epsilon  a_{1} & 1+
			\epsilon \delta
			& 0 
			\\
			0 & \epsilon  b_{1} & 1+\epsilon  b_{2}
			\\
			0 & 0 & \epsilon  c_{1} 
		\end{pmatrix}
		\begin{pmatrix}
			{x_{1}}
			\\
			{x_{2}}
			\\
			{x_{3}}
		\end{pmatrix}+\begin{pmatrix}
		a^{(3)}_{5} x_2^2-
		a^{(3)}_{7} (2{x_{1}}{x_{3}}-x_2^2)
			\\
 b_4^{(2)} x_2^{2}
			\\
	 c_2x_3^{2}
		\end{pmatrix},
  \eas
	where the coefficients  \( a^{(3)}_{5}, a^{(3)}_{7},\) and \( b_4^{(2)}\) are   introduced  in  \eqref{eq:nfcoeff} and \eqref{eq:b2} respectively.
	\end{theorem}
	\begin{remark}

 If one has a transvectant, one can choose any of these terms as the term in the normal form, depending on one's preference. For example, one can replace $(x_2^{2}-2{x_{1}} {{x_{3}}})$ with either ${x_{1}} {x_{3}}$ or $x_2^2,$ but not with \(\mathsf{L}_{\N}({x_{1}}{x_{2}})=x_2^2+{x_{1}}{x_{3}}.\) While this works for low degrees, it is not so easy to do this consistently for higher degrees.
		
	\end{remark}

	\subsection{Examples}\label{subsec:examples}
In this section, we present normal forms using the triangular \(\Sl\)-style, which was introduced earlier, for feed forward systems with dimensions ranging from \(2\) to \(7\), all of which exhibit near-nilpotent singularities. Mainly, we provide the symbolic version of normal forms and utilize Theorem \eqref{thm:gen} and Corollary \eqref{cor:app} to validate the normal form. In these examples, \(a^i_j\) for all \(i = 1, \ldots, 4\) and \(j = 1, \ldots, 7\) are the coefficients of the normal forms, which are functions of the coefficients of the original system.

		\begin{itemize} \label{eq:nf2d}
  \item \(2\)D.
			A simple example is  the   nilpotent feed forward:
			\ba\label{eq:n2}
			\left\{ \begin{array}{l} \dot {x}_{1} = 
				{x_{2}}+ F_1({x_{1}},{x_{2}}),\\ 
				\dot {x}_{2} =F_2({x_{2}}).
			\end{array} \right.
   \ea
   The \(\Sl\)-triple in the first component is given by:
   \ba\label{eq:sl2v2}
   \nonumber
   {\N}&=& {x_{2}} \frac{\partial}{\partial {x_{1}}},
   \\\nonumber
   {\M}&=& {x_{1}} \frac{\partial}{\partial {x_{2}}},
   \\
    {\H}&=&  {x_{1}} \frac{\partial}{\partial {x_{1}}}-{x_{2}} \frac{\partial}{\partial {x_{2}}}.
   \ea
	
   Following  Theorem \eqref{thm:gen} the generating function of normal form  for \(2\)D feed forward system  is \(t^2u^2.\)  Which is equal to \(\mathcal{P}^{1,1}_2(u).\)
   Hence,  the quadratic term
   in \({\ker}(\mathsf{L}_{{\M}})\) is \(\tau^0_1=(x_1,x_1)^{(0)}=x_1^2\) with \(\omega_{\tau^0_2}=2.\) 
   The normal form is
			\bas
			\left\{ \begin{array}{l} \dot {x}_{1} = 
				\mu_1 {x_{1}}+{x_{2}}+ a^1_1x_1^2,\\ 
				\dot {x}_{2} =\nu+\mu_2 {x_{2}}+a^2_1x_2^2.
			\end{array} \right.
			\eas

   \item \(3\)D. This case is already treated in the former subsection. 
				\ba\label{eq:3d}
				\left\{ \begin{array}{l} 
					\dot {x}_{1}=  {x_{2}}+ F_1({x_{1}},{x_{2}},{x_{3}}),
					\\
					\dot {x}_{2}= {x_{3}}+ F_2({x_{2}},{x_{3}}),
					\\
					\dot {x}_{3}= F_3({x_{3}}).
				\end{array} \right.
				\ea
		The \(\Sl\)-triple  in the  first  component is as 
				\bas 
				\N&=&  {x_{2}}\frac{\partial}{\partial {x_{1}}}+{x_{3}}\frac{\partial}{\partial {x_{2}}},
				\\
				\M&=& 2{x_{1}}\frac{\partial}{\partial {x_{2}}}+2 {x_{2}} \frac{\partial}{\partial {x_{3}}},
				\\
				\H&=& 2 {x_{1}}  \frac{\partial}{\partial {x_{1}}}-2 {x_{3}}  \frac{\partial}{\partial {x_{3}}}.
				\eas
		
				Using Lemma \ref{lem:trans} we find \(\tau^0_1\) and \(\tau^2_1\), associated with the term \(d^2\) \((1+u^4)t^2\) in \(\mathcal{P}_2(d,u)\),
			\bas 
			(x_1,x_1)^{(0)}=&\tau^0_1=x_1^2,\quad\quad\quad\quad\,\,\,
		\,\,\,\,&\omega_{\tau^0_1}=4,
			\\
			{(x_1,x_1)}^{(2)}=&\tau^2_1=2{x_{1}}{x_{3}}-x_2^2,\,\,\,
			\,\,\,\,&\omega_{\tau^2_1}=0.
			\eas
   The quadratic normal form is:
				\ba
				\left\{ \begin{array}{l} \dot {x}_{1} = 
					{x_{2}} +\mu_1 {x_{1}}+a^1_1 x_1^2+a^1_2(2{x_{1}} {x_{3}}-x_2^2),\\ 
					\dot {x}_{2} ={x_{3}}+\mu_2 {x_{2}}+a^2_1 x_2^2,
					\\\label{eq:nf3}
					\dot {x}_{3} =\nu+\mu_3 {x_{3}}+a^3_1x_3^2.
				\end{array} \right.
				\ea
	
				\item \(4\)D. Consider the following feed forward system
    \bas 
				\left \{ \begin{array}{l} 
					\dot {x}_{1}=  {x_{2}}+ F_1({x_{1}},{x_{2}},{x_{3}},x_4),
					\\
					\dot {x}_{2}= {x_{3}}+ F_2({x_{2}},{x_{3}},x_4),
					\\
					\dot {x}_{3}= x_4+ F_3({x_{3}},x_4),
     \\
     \dot x_4=F_4(x_4),
				\end{array} \right.
				\eas
    with the following \(\Sl\)-triple:
    \bas 
				\N&=&  x_2\frac{\partial}{\partial {x_{1}}}+x_3\frac{\partial}{\partial {x_{2}}}+x_4\frac{\partial}{\partial {x_{3}}},
				\\
				\M&=& 3\,x_{{1}}\frac{\partial}{\partial {x_{2}}}+4\,x_2\frac{\partial}{\partial {x_{3}}}+3\,x_3\frac{\partial}{\partial x_4},				\\
				\H&=&3\,x_{{1}}\frac{\partial}{\partial {x_{1}}}+x_2\frac{\partial}{\partial {x_{2}}}-
x_3\frac{\partial}{\partial {x_{3}}}-3\,x_4\frac{\partial}{\partial x_4}.
				\eas

   Using Lemma  \ref{lem:trans} we see that \(\tau^0_1\) and \(\tau^2_1\), associated to the \(d^3\)-term \(u^2d\cdot d^2t^2+ (u^2d)^3t^2=u^2(1+u^4)t^2 d^3\) in \(\mathcal{P}_2(d,u)\), which are defined by 
			\bas 
			(x_1,x_1)^{(0)}=\tau^0_1=&x_1^2,\,\,\,
			\,\,\,\,\,\,\,\,\qquad&\omega_{\tau^0_1}=6,
			\\
			(x_1,x_1)^{(2)}=\tau^2_1=&6{x_{1}}{x_{3}}-4x_2^2,\,\,\,
			&\omega_{\tau^2_1}=2.
			\eas
   The normal form is given by:
				\bas
				\left\{ \begin{array}{l}
					\dot {x}_{1} = {x_{2}}+\mu_1 {x_{1}}+a^1_1x_1^2+a^1_2(6{x_{1}}{x_{3}}-4x_2^2),
					\\
					\dot {x}_{2} = {x_{3}}+\mu_2 {x_{2}}+
					a^2_1 x_2^2+a^2_2(2{x_{2}} x_4-x_3^2),\\ 
					\dot {x}_{3} =x_4+\mu_3 {x_{3}}+a^3_1 x_3^2,
					\\ 
					\dot x_4 =\nu+\mu_4 x_4+a^4_1x_4^2.
				\end{array} \right.
				\eas
			\item \(5\)D.
			The \(\Sl\)-triple for dimension \(5\)
			\bas
			\N&=&\,x_2\frac{\partial}{\partial {x_{1}}}+x_3\,\frac{\partial}{\partial {x_{2}}}+\,x_4
			\frac{\partial}{\partial {x_{3}}}+x_5\frac{\partial}{\partial x_4},
			\\
			\M&=&4x_{{1}}\frac{\partial}{\partial {x_{2}}}+6x_2
			\,\frac{\partial}{\partial {x_{3}}}+6x_3\,\frac{\partial}{\partial x_4}+4x_{{
					4}}\,\frac{\partial}{\partial x_5},
			\\
			\H&=&4\,
			x_{{1}}\frac{\partial}{\partial {x_{1}}}+2x_2\,\frac{\partial}{\partial {x_{2}}}-2\,x_4\frac{\partial}{\partial x_4}-4\,x_5\frac{\partial}{\partial x_5}.
			\eas
			The generating function for  \({\ker}(\mathsf{L}_{{\M}})\) is given by
			\(
			(u^8+u^4+1)t^2
			\), which is equal to \(\mathcal{P}^{5,1}_2(u)\),
   see Theorem \ref{thm:gen}.
   Using Lemma  \ref{lem:trans} we see that 
   \({\ker}(\mathsf{L}_{{\M}})\) is  given by \(\tau^0_1,\tau^2_1\) and \(\tau^4_1\) which are defined by 
			\bas 
			(x_1,x_1)^{(0)}&=&\tau^0_1=x_1^2,\,\,\,\,\qquad\qquad\qquad\qquad
			\,\,\,\,\,\,\omega_{\tau^0_1}=8,
			\\
			(x_1,x_1)^{(2)}&=&\tau^2_1=12 x_{1} x_{3}-9 x_{2}^{2},\,\,\,\,\,\qquad\qquad
			\omega_{\tau^2_1}=4,
			\\
			(x_1,x_1)^{(4)}&=&\tau^4_1=2x_{{1}}x_5-2\,x_2x_4+\,x_3^{2},\,\,\,\,\,\quad
			\omega_{\tau^4_1}=0.
			\eas
		Hence, the normal form is
			\bas
			\left\{ \begin{array}{l}
				\dot {x}_{1}={x_{2}}+\mu_1 {x_{1}}+a^1_1 x_1^2+a^1_2
				(12{x_{1}}{x_{3}}-9x_2^2)+a^1_3
	\left(2x_{{1}}x_5-2\,x_2x_4+\,x_3^{2}\right),
				\\
				\dot {x}_{2} = {x_{3}}+\mu_2 {x_{2}}+a^2_1x_2^2+a^2_2(6{x_{2}}x_4-4x_3^2),
				\\
				\dot {x}_{3} = 
				x_4 +\mu_3{x_{3}}+a^3_1 x_3^2+a^3_2(2{x_{3}} x_5-x_4^2),\\ 
				\dot x_4 =x_5+\mu_4 x_4+a^4_1 x_4^2,
				\\ 
				\dot x_5 =\nu+\mu_5 x_5+a^5_1x_5^2.
			\end{array} \right.
			\eas

			\item \(6\)D.
			The \(\Sl\)-triple for dimension \(6\) is as 
			\bas 
			\N&=&\,x_2\frac{\partial}{\partial {x_{1}}}+\,x_3\frac{\partial}{\partial {x_{2}}}+\,x_4
			\frac{\partial}{\partial {x_{3}}}+\,x_5\frac{\partial}{\partial x_4}+x_6\frac{\partial}{\partial x_5},
			\\
			\M&=&
			5 x_{{1}}\frac{\partial}{\partial {x_{2}}}+8\,x_2\frac{\partial}{\partial {x_{3}}}+9\,x_3
			\frac{\partial}{\partial x_4}+8\,x_4\frac{\partial}{\partial x_5}+5\,x_5\frac{\partial}{\partial x_6},
			\\
			\H&=& 5\,x_{{1
			}}\frac{\partial}{\partial {x_{1}}}+3\,x_2\frac{\partial}{\partial {x_{2}}}+x_3\frac{\partial}{\partial {x_{3}}}-x_{{4}
			}\frac{\partial}{\partial x_4}-3\,x_5\frac{\partial}{\partial x_5}-5\,x_6\frac{\partial}{\partial x_6}.
			\eas
The invariants  give us a generating function
   \(t^2(u^{10}+u^{6}+u^2)\) see Theorem \ref{thm:gen}, which is equal to \(\mathcal{P}^{6,1}_2(u).\)
   Using Lemma \eqref{lem:trans} we observe  that  the quadratic terms in
   \({\ker}(\mathsf{L}_{{\M}})\)  are given by
			\bas 
			(x_1,x_1)^{(0)}=&\tau^0_1=x_1^2,\,\,\,
			\qquad\qquad\qquad\qquad\,\,\,\,\,\,\,\,\,\,\,\,\quad&\omega_{\tau^0_1}=10,
			\\
			(x_1,x_1)^{(2)}=&\tau^2_1=20{x_{1}}{x_{3}}-16x_2^2,\,\,\qquad\qquad\qquad
			&\omega_{\tau^2_1}=6,
			\\
			(x_1,x_1)^{(4)}=&\tau^4_1=10x_{{1}}x_5-16\,x_2x_4+9\,x_3^{2},\,
			\qquad&\omega_{\tau^4_1}=2,
			\eas
which generates the terms of normal form. Therefore,
			the normal form is 
			\bas
			\left\{ \begin{array}{l}
				\dot {x}_{1}={x_{2}}+\mu_1 {x_{1}}+a^1_1 x_1^2+a^1_2
				(20{x_{1}}{x_{3}}-16x_2^2)+a^1_3
				\left(10x_{{1}}x_5-16\,x_2x_4+9\,x_3^{2}\right),
				\\
				\dot {x}_{2}={x_{3}}+\mu_2 {x_{2}}+a^2_1 x_2^2+a^2_2
				(12{x_{2}}x_4-9x_3^2)+a^2_3
				\left(2x_2x_6-2\,x_3x_5+\,x_4^{2}\right),
				\\
				\dot {x}_{3} =x_4+\mu_3 {x_{3}}+a^3_1x_3^2+a^3_2(6{x_{3}}x_5-4x_4^2),
				\\
				\dot x_4 =
				x_5 +\mu_4 x_4+a^4_1 x_4^2+a^4_2(2x_4 x_6-x_5^2),\\ 
				\dot x_5 =x_6+\mu_5 x_5+a^5_1 x_5^2,
				\\ 
				\dot x_6 =\nu+\mu_6 x_6+a^6_1x_6^2.
			\end{array} \right.
			\eas

			\item \(7\)D.
			The \(\Sl\)-triple for dimension \(7\) is 
			\bas 
			{\N}&=&\,x_2\frac{\partial}{\partial {x_{1}}}+\,x_3\frac{\partial}{\partial {x_{2}}}+
			\,x_4
			\frac{\partial}{\partial {x_{3}}}+\,x_5\frac{\partial}{\partial x_4}+\,x_6\frac{\partial}{\partial x_5}+x_7\frac{\partial}{\partial x_6},
			\\
			{\M}&=&6x_{{1}}\frac{\partial}{\partial {x_{2}}}+10\,x_2\frac{\partial}{\partial {x_{3}}}+12\,x_3\frac{\partial}{\partial x_4}+12\,x_4\frac{\partial}{\partial x_5}+10\,x_5
			\frac{\partial}{\partial x_6}+6\,x_6\frac{\partial}{\partial x_7},
			\\
			\H&=&6\,x_{{1}}\frac{\partial}{\partial {x_{1}}}+4\,
			x_2\frac{\partial}{\partial {x_{2}}}+2\,x_3\frac{\partial}{\partial {x_{3}}}-2\,x_5\frac{\partial}{\partial x_5}
			-4\,x_6\frac{\partial}{\partial x_6}-6\,x_7\frac{\partial}{\partial x_7}.
			\eas
  The generating function for this dimension gives us
   \(t^2(u^{12}+u^{8}+u^4+1)\), see Theorem \ref{thm:gen}, which is equal to \(\mathcal{P}^{7,1}_2(u).\)
   Using Lemma \ref{lem:trans} we see that 
   \({\ker}(\mathsf{L}_{{\M}})\) is  given by \(\tau^0_1,\tau^2_1,\tau^4_1\) and 
   \(\tau^6_1\) which are defined by these transvectants  
			\bas 
			(x_1,x_1)^{(0)}=&\tau^0_1=x_1^2,\,\,\,\qquad\qquad\qquad\qquad\qquad\qquad\qquad
			&\omega_{\tau^0_1}=12,
			\\
				(x_1,x_1)^{(2)}=&\tau^2_1=20{x_{1}}{x_{3}}-16x_2^2,
			\,\,\,\quad\quad\quad\quad\quad\quad\quad\quad\,\,\,\,&\omega_{\tau^2_1}=8,
			\\	(x_1,x_1)^{(4)}=&\tau^4_1
   =30\,x_{{1}}x_5-60\,x_2x_4+36\,x_3^{2},\,\,\,
			\qquad\quad\quad&\omega_{\tau^4_1}=4,
			\\
			(x_1,x_1)^{(6)}=&\tau^6_1=
		2\,x_{{1}}x_7-2\,x_2x_6+2\,x_3x_5-x_4^{2},
\,\,\,\,\,\qquad&\omega_{\tau^6_1}=0.
			\eas
			The normal form is 
			\bas
			\left\{ \begin{array}{l}
				\dot {x}_{1}={x_{2}}+\mu_1 {x_{1}}+a^1_1 x_1^2+a^1_2
				(20{x_{1}}{x_{3}}-16x_2^2)+a^1_3
    \left(30\,x_{{1}}x_5-60\,x_2x_4+36\,x_3^{2}\right)
    \\\,\,\,\,\,\,\,\,\,\,\,\,+
a^1_4\left(2\,x_{{1}}x_7-2\,x_2x_6+2\,x_3x_5-x_4^{2}\right),
				\\
				\dot {x}_{2}={x_{3}}+\mu_2 {x_{2}}+a^2_1 x_2^2+a^2_2
				({x_{2}}x_4-x_3^2)+a^2_3
	\left(x_2x_6-4\,x_3x_5+3\,x_4^{2}\right),
				\\
				\dot {x}_{3}=x_4+\mu_3 {x_{3}}+a^3_1 x_3^2+a^3_2
				({x_{3}}x_5-x_4^2)+a^3_3
	\left(x_3x_7-4\,x_4x_6+3\,x_5^{2}\right),
				\\
				\dot x_4 =x_5+\mu_4 x_4+a^4_1x_4^2+a^4_2(x_4x_6-x_5^2),
				\\
				\dot x_5 = 
				x_6+\mu_5 x_5 +a^5_1 x_5^2+a^5_2(x_5 x_7-x_6^2),\\ 
				\dot x_6 =x_7+\mu_6 x_6+a^6_1 x_6^2,
				\\ 
				\dot x_7 =\nu+\mu_7 x_7+a^7_1x_7^2.
			\end{array} \right.
			\eas
	
		\end{itemize}

 \section{Outer normal form}\label{sec:orbital}

\subsection{General outer normal form theory for block-triangular systems} \label{sec:outerNF}
The orbital normal form is a type of normal form that allows the use of functions as transformations for further simplification in the vicinity of original fixed points. This method shares similarities with the right- or left-equivariant approach, as shown in \cite{golubitsky2012singularities,meyer1986singularities}. Previous works \cite{gazor2013normal,bogdanov1979local,bogdanov1976local,chen2005unique,strozyna2003orbital} have already utilized the orbital normal form.
In this paper, we extend the concept of orbital normal form to that of outer normal form for block-triangular systems.

In this section, we first demonstrate that the action of the orbital normal form can be represented. In the two subsequent subsections, we apply the orbital normal form to the feed forward system near the nilpotent. Notably, through this operator, we can eliminate all terms near the quadratic terms, significantly simplifying the expressions.

	The orbital normal form is well known in the normal form literature,
 but it seems to be lacking a Lie algebraic foundation. This is what we want to provide here.
 The proof will work for the general case, where one multiplies the whole equation with one function, but is written in such a way that we have more freedom in the triangular case, where we multiply each component with a different function.
This can be generalized to intermediate cases, which we call {\em block-triangular} as follows.
\def\g{x}

\subsubsection{Block-triangular systems and their properties}\label{sec:blocktriangular}
 Suppose we can write the vector fields \(X\) we are considering as
 \(\blockX=\sum_{i=1}^m \blockX^i\), where
 \bas
 \blockX^i=\sum_{j=p_i}^{p_{i+1}-1} X_j\partial^j\in\mathfrak{\g}^i,
 \eas
\(X_j\in\mathfrak{r}^i:=\mathbb{R}[x_{p_i},\cdots,x_n]\) for \(j=p_i,\ldots,p_{i+1}-1\), and \(1=p_1<\cdots<p_m<p_{m+1}=n+1\). 
We see that \(\mathfrak{r}^1\supset\cdots\supset\mathfrak{r}^{m}\).
If \(m=1\), then

 \bas
 \blockX^1=\sum_{j=1}^{n} X_j(x_{1},\cdots,x_n)\partial^j,
 \eas
 describing all polynomial vector fields on \(\mathbb{R}^n\). In this paper, we are interested in the case \(m=n\), that is, \(p_i=i, i=1,\ldots,n\). In this case, \(\mathfrak{\g}_i\) corresponds to \(\mathfrak{\g}^{i}\) in \S\ref{sec:algorithm}. So block-triangular describes both the ordinary differential equations, without restrictions, as well
 as the triangular systems.

Notice that, with mintropical addition \(i\oplus j=\min(i,j)\),
\(\mathfrak{r}^i\otimes\mathfrak{r}^j=\mathfrak{r}^{\min(i,j)}=\mathfrak{r}^{i\oplus j}\),
giving us a (tropical) grading on the \(\mathfrak{r}^i\). This generalizes Proposition \ref{pro:lie} to the case \(m\leq n-1\).

The commutator of \(\blockX^r\) and \(\blockY^s\) is
\bas
[\blockX^r,\blockY^s]&=&[\sum_{i=p_r}^{p_{r+1}-1} X_i(x_{p_r},\cdots,x_n)\partial^i,\sum_{j=p_s}^{p_{s+1}-1} Y_j(x_{p_s},\cdots,x_n)\partial^j]
\\&=&
\sum_{i=p_r}^{p_{r+1}-1}\sum_{j=p_s}^{p_{s+1}-1} \left( X_i\frac{\partial Y_j}{\partial x_i}\partial^j 
-Y_j\frac{\partial X_i}{\partial x_j}\partial^i\right).
\eas

Assume \(r<s\). Then \(i\leq p_{r+1}-1\leq p_s-1< p_s\).
Then \(\frac{\partial Y_j(x_{p_s},\cdots,x_n)}{\partial x_i}=0\) and
\bas
[\blockX^r,\blockY^s]&=&
-\sum_{i=p_r}^{p_{r+1}-1}\sum_{j=p_s}^{p_{s+1}-1} Y_j(x_{p_s},\cdots,x_n)\frac{\partial X_i(x_{p_r},\cdots,x_n)}{\partial x_j}\partial^i
\in\mathfrak{\g}^r.
\eas
In general, with mintropical addition \(r\oplus s=\min(r,s)\),
\bas
[\mathfrak{\g}^r,\mathfrak{\g}^s]\subset\mathfrak{\g}^{r\oplus s},
\eas
giving us a (tropical) grading on the \(\mathfrak{\g}^i\).

Let \(\mathfrak{t}^i=\bigoplus_{j=1}^i\mathfrak{\g}^j\). 
Then
\[[\mathfrak{t}^i,\mathfrak{t}^j]\subset\bigoplus_{k=1}^i \bigoplus_{l=1}^j [\mathfrak{\g}^k,\mathfrak{\g}^l]\subset 
\bigoplus_{k=1}^i \bigoplus_{l=1}^j \mathfrak{\g}^{\min(k,l)}=\bigoplus_{l=1}^{\min(i,j)}\mathfrak{\g}^l=\mathfrak{t}^{i\oplus j },\]
and, in particular \([\mathfrak{t}^m,\mathfrak{\g}^j]\in\mathfrak{t}^j\).
It is this property that is the main motivation for the definition of block-triangular.

We say that \(\mathfrak{t}^i\) is a Lie ring, with ring \(\mathfrak{r}^i=\mathbb{R}[x_{p_i},\ldots,x_n]\). 
We now introduce the
action of the ring \(\mathfrak{r}^i\) on \(\mathfrak{t}^i\) as follows.
Let \(r^i\in\mathfrak{r}^i\), \(\blockY^i\in\mathfrak{\g}^i\), and put
\bas
r^i \star  \blockY^i:=\sum_{j=p_i}^{p_{i+1}-1} r^i Y_j\partial^j\in\mathfrak{\g}^{i\oplus i}=\mathfrak{\g}^{i}.
\eas

Then, with \(\blockX\in\mathfrak{t}^m\) and \(\blockY^i\in\mathfrak{\g}^i\),  
one has with $[\blockX,\blockY^i]\in\mathfrak{t}^i$ and \(X(\mathfrak{r}^i)\subset \mathfrak{r}^i\),
\bas
[\blockX,r^i\star \blockY^i]&=&\sum_{l=1}^n\sum_{j=n-p_i}^{n-p_{i-1}-1}  \left( X_l\frac{\partial r^i Y_j}{\partial x_l}\partial^j-r^i Y_j\frac{\partial X_l}{\partial x_j}\partial^l\right)
\\&=&
\sum_{l=1}^n\sum_{j=n-p_i}^{n-p_{i-1}-1}  \left( X_l\frac{\partial r^i }{\partial x_l}Y_j\partial^j+r^iX_l\frac{\partial Y_j}{\partial x_l}\partial^j-r^i Y_j\frac{\partial X_l}{\partial x_j}\partial^l\right)
\\&=&
\sum_{l=1}^n X_l\frac{\partial r^i }{\partial x_l} \sum_{j=n-p_i}^{n-p_{i-1}-1}Y_j\partial^j+ r^i\sum_{l=1}^n\sum_{j=n-p_i}^{n-p_{i-1}-1}   \left( X_l\frac{\partial Y_j}{\partial x_l}\partial^j- Y_j\frac{\partial X_l}{\partial x_j}\partial^l\right)
\\&=&
\blockX(r^i)\star \blockY^i+r^i\star [\blockX,\blockY^i]\in\mathfrak{t}^i.
\eas

 The following theorem is a standard Lie algebraic result. 
 \begin{theorem}\label{thm:ro}
 Let \(\mathfrak{t}\) be a Lie algebra and  
 let \(\mathfrak{r}\) be a representation space of \(\mathfrak{t}\). Then $\mathfrak{g}=\mathfrak{r}\rtimes \mathfrak{t},$  with bracket 
    \bas
	[(r_1,{\blockX^{1}}),(r_2,{\blockX^{2}})]_{\mathfrak{g}}=\left({\blockX^{1}}r_2-{\blockX^{2}}r_1,[{\blockX^{1}},{\blockX^{2}}]_{\mathfrak{h}}\right),
	\eas
is the {\em(}trivial{\em)} abelian extension of \({\mathfrak{t}}\) by 
\({\mathfrak{r}}\), and
defines a Lie algebra.
 \end{theorem}
\subsubsection{Outer automorphisms}\label{sec:outomorphisms} 
The following theorem is exactly what one needs to do in normal form theory with {\em outer automorphisms}.
 Transformations are called inner when the generators are in the representation space, that is, they are coordinate transformations, otherwise, they are called outer. For instance, in the triangular case, completely general coordinate transformations could be seen as outer, as long as they produce a triangular normal form.
 It is not always clear that the outer is more than the inner; take, for example, simple Lie algebras (Whitehead lemmas). But the Lie algebras in nonlinear normal form theory are not simple.

The fact that orbital normal form (that is, with \(p=n\)) is working successfully in the cited literature, is an indication that outer is indeed often better.

\begin{theorem}\label{thm:oop}
Let \(\mathfrak{h}\subset\mathfrak{t}\) be an ideal and suppose \(\mathfrak{r}\) acts on \(\mathfrak{h}\) with the \(\star\)-action
and corresponding rule \bas [\blockX,r\star \blockY]=\blockX(r)\star \blockY +r \star [\blockX,\blockY]_\mathfrak{t}\in \mathfrak{h},\eas
where \((r,\blockX)\in \mathfrak{g}=\mathfrak{r}\rtimes \mathfrak{t}, \blockY\in \mathfrak{h}\), the semidirect product of \(
\mathfrak{r}\) and \(\mathfrak{t}\), with a given action of  \(\mathfrak{t}\) on \(\mathfrak{r}\), cf. {\em Theorem \ref{thm:ro}}.
Define 
\ba\label{eq:RHOOP}
\rho(r,\blockX)\blockY=r\star \blockY+[\blockX,\blockY]_{\mathfrak{t}}.
\ea
Then, \(\rho\)  is a representation of \(\mathfrak{g}\) in \(\mathfrak{h}.\)
 \end{theorem}
 \begin{proof}
 In order to prove that \(\rho\) is the representation of
 $\mathfrak{g}$ in \({\rm End{(\mathfrak{h}})}\),
 we have to show that the following holds:
     \bas
{\rm \rho}([(r_1,{\blockX^{1}}),(r_2,{\blockX^{2}})]_{\mathfrak{g}})=[{\rm \rho}(r_1,{\blockX^{1}}),{\rm \rho}(r_2,{\blockX^{2}})]_{\rm End{(\mathfrak{h}})}.
		\eas

 By computing the expression on the right-hand side we find, with \(\blockV\in\mathfrak{h}\),
 \bas
		[\rho(r_1,{\blockX^{1}}),\rho(r_2,{\blockX^{2}})]_{\rm End{(\mathfrak{t}})}\blockV
		&=&	\rho(r_1,{\blockX^{1}})\rho(r_2,{\blockX^{2}})\blockV
		-
		\rho(r_2,{\blockX^{2}})\rho(r_1,{\blockX^{1}})\blockV
		\\
		&=&	\rho(r_1,{\blockX^{1}})(r_2\star \blockV+[{\blockX^{2}},\blockV]_{\mathfrak{t}})
		-
		\rho(r_2,{\blockX^{2}})(r_1\star \blockV+[{\blockX^{1}},\blockV]_{\mathfrak{t}})
		\\
		&=& r_1\star r_2\star \blockV+r_1\star [{\blockX^{2}},\blockV]_{\mathfrak{t}}+[{\blockX^{1}},r_2\star \blockV]_{\mathfrak{t}}+[{\blockX^{1}},[{\blockX^{2}},\blockV]_{\mathfrak{t}}]_{\mathfrak{t}}
  \\&&
  -
	r_2\star r_1\star \blockV-r_2\star [{\blockX^{1}},\blockV]_{\mathfrak{t}}-[{\blockX^{2}},r_1\star \blockV]_{\mathfrak{t}}-[{\blockX^{2}},[{\blockX^{1}},\blockV]_{\mathfrak{t}}]_{\mathfrak{t}},
  \eas
 since $\mathfrak{r}$
  is a commutative ring we find 
  \bas
 	\lefteqn{[\rho(r_1,{\blockX^{1}}),\rho(r_2,{\blockX^{2}})]_{\rm End{(\mathfrak{h}})}\blockV
		=}&&
\\&=& r_1\star [{\blockX^{2}},\blockV]_{\mathfrak{t}}+[{\blockX^{1}},r_2\star \blockV]_{\mathfrak{t}}+[{\blockX^{1}},[{\blockX^{2}},\blockV]_{\mathfrak{t}}]_{\mathfrak{t}}-r_2\star [{\blockX^{1}},\blockV]_{\mathfrak{t}}-[{\blockX^{2}},r_1\star \blockV]_{\mathfrak{t}}-[{\blockX^{2}},[{\blockX^{1}},\blockV]_{\mathfrak{t}}]_{\mathfrak{t}}.
  \eas
  By applying the identity  
  	$$[\blockX,r\star \blockV]=\blockX(r)\star \blockV+r\star [\blockX,\blockV],$$
  and the Jacobi identity we get 
  
  \bas
  [\rho(r_1,{\blockX^{1}}),\rho(r_2,{\blockX^{2}})]_{\rm End{(\mathfrak{h}})}\blockV&=&-{\blockX^{2}}(r_1)\star \blockV+{\blockX^{1}}(r_2)\star \blockV
 +[{\blockX^{1}},[{\blockX^{2}},\blockV]_{\mathfrak{t}}]
		-[{\blockX^{2}},[{\blockX^{1}},\blockV]_{\mathfrak{t}}]_{\mathfrak{t}}
  \\&=&
  ({\blockX^{1}}(r_2)-{\blockX^{2}}(r_1))\star \blockV
+[[{\blockX^{1}},{\blockX^{2}}]_{\mathfrak{t}},\blockV]_{\mathfrak{t}}
\\
&=&
\rho( ({\blockX^{1}}(r_2)-{\blockX^{2}}(r_1),[{\blockX^{1}},{\blockX^{2}}]_{\mathfrak{t}}))\blockV
\\
&=&
\rho([(r_1,{\blockX^{1}}),(r_2, {\blockX^{2}})]_{\mathfrak{g}}) \blockV.
	\eas

	This implies that we can generate transformations with elements in \(\mathfrak{g}\) to transform elements in \(\mathfrak{h}\).
If we let \(\mathfrak{h}=\mathfrak{t}\), then the transformations are outer automorphisms of \(\mathfrak{t}\).
\end{proof}
\subsubsection{Computation of the outer normal form}\label{sec:otnft}
Suppose we are given a bigraded set of subspaces \(\mathfrak{\g}^i_k\) of a bigraded Lie ring, such that
\[
[\mathfrak{\g}^i_k,\mathfrak{\g}^j_l]\subset\mathfrak{\g}^{i\oplus j}_{k+l},\quad \mathfrak{r}^i_k\mathfrak{\g}^i_l\subset\mathfrak{\g}^i_{k+l}.
\]
Let \[\mathfrak{t}^i_j=\bigoplus_{k=1}^i \bigoplus_{l=0}^j \mathfrak{\g}^l_k.\]
Then \(\mathfrak{t}=\mathfrak{t}^m\supset\mathfrak{t}^{m-1}\supset\cdots\supset\mathfrak{t}^1\supset\mathfrak{t}^0=0\) is a finite min-tropical filtration, for vector fields induced by the block structure,
and \(\mathfrak{t}^i=\mathfrak{t}^i_0\supset\mathfrak{t}^i_1\supset\cdots\) is an infinite filtration, for vector fields induced by the grade. Let \(\pi^i_j:\mathfrak{t}\rightarrow\mathfrak{\g}^i_j\) be the natural projection.
The normal form procedure is as follows:

We are given an object \(V^{[m]}_{[0]}\in\mathfrak{t}\) and we want to bring it in some kind of normal form using outer transformations.
In order to solve the normal form equation, we need a part of \(V^{[m]}_{[0]}\) with invertible coefficients.
We can for instance start by taking the lowest order part \(V^{(1)}_{(0)}=V^{[1]}_{[0]}\mod\mathfrak{t}^1_1\) and from \(\mathsf{V^{(1)}_{(0)}}\) remove the non-invertible coefficients.

In general, we take
 \(1\leq i \leq m\) and \( 1\leq {j_i}< \infty\).
Assume \(\mathfrak{n}^k_l\), the complement to the \(\rho\)-action in \(\mathfrak{\g}^k_l\), is determined for \(k\leq i\) and \(1\leq l< j_i\). We initially take \(i=1\) and \(j_1=1\), and make no assumptions on the terms being in normal form. We now describe the general step, going from
\((i,j_i-1)\) to \((i,j_i)\). We assume \(V^{[i]}_{[j_i-1]}\) to be in normal form up to the corresponding filtration degrees.
Then,
in order to solve the normal form equation, we need a part of \(V^{[i]}_{[j_i-1]} \) with invertible coefficients.
We can for instance take the lowest order part \(V^{(i)}_{(j_i-1)}=V^{[i]}_{[j_i-1]}\mod\mathfrak{t}^i_{j_i}\) and from \(\mathsf{V}^{(i)}_{(j_i-1)}\) remove the non-invertible coefficients.


We now take our transformation generators from \(Y^i_{j_i}\in\mathfrak{r}^i\rtimes\mathfrak{t}^i_1\), such that
\[
\rho(Y^i_{j_i})V^{(i)}_{({j_i-1})}\in\mathfrak{t}^i_{j_i},
\]
making sure the transformation does not disturb the terms already in normal form, and
determine 

\(\mathrm{im}\ \rho(\mathfrak{r}^i\rtimes\mathfrak{t}^i_1)V^{(i)}_{({j_i}-1)}\) and choose a complement \(\mathfrak{n}^i_{j_i}\)
to it in \(\mathfrak{t}^i_{j_i}/\mathfrak{t}^i_{{j_i}+1}\).

Then let \(V^i_{j_i}=\pi^i_j V^{[{i}]}_{[j_i-1]}\) and compute a transformation generator \(Y^i_{j_i}\in\mathfrak{r}^i\rtimes\mathfrak{t}^i_1\), such that 
\[
V^{i}_{j_i}+\rho(Y^i_{j_i})V^{(i)}_{({j_i})}\in\mathfrak{n}^i_{j_i}.
\]
Let 
\bas
V^{[i]}_{[{j_i}]}&=&\exp\rho(Y^i_{j_i}) V^{[i]}_{[{j_i}-1]}
\\&=&V^{[i]}_{[j_i-1]}+ \rho(Y^i_{j_i})V^{[i]}_{[{j_i}]}+\cdots
\eas
The projection on \(\mathfrak{\g}^i_{j_i}\) equals \(V^{i}_{j_i}+\rho(Y^{(i)}_{(j_i)}) V^{(i)}_{(j_i)}\in\mathfrak{n}^i_{j_i} \)
and we have \(V^{[i]}_{[{j_i}]}\) in normal form with respect to \(V^{(i)}_{({j_i}-1)}\).

This allows us to increase \(j_i\). This can go on indefinitely; in practice, in a computation, one puts an upper bound on the \(j_i\)s; nevertheless the description of the normal form might be possible to \(\infty\), at least for formal power series. As we will show in 
sections \ref{sec:2dorbital} and \ref{sec:3dorbital}.
To increase \(i\), we put \(j_{i+1}=1\) and start all over again.

Remark that while for vector fields the grading is the degree minus one, for the elements in \((r,\blockX)\in\mathfrak{g}\) the grading of
\(r\) is the degree and for \(\blockX\) the usual degree minus one.
Also, one should be careful with the interpretation of the result, which looks like an ordinary vector field, 
However, since every component could have been multiplied by a different function, time means something different for every component.
Since the transformation \(\exp(\rho(r,\blockX))\) is invertible and acts on the coordinates, this should be used to recover
useful approximating solutions with the correct uniform timing.

  Finally, we mention that if \(\mathsf{S}\neq 0\), it makes sense to choose \(r\in\ker\mathsf{S}\), if the vector field is in \(\ker\ad{\mathsf{S}}\).

\subsection{Quadratic outer normal form  in the triangular case \(m=n-1\)}
\def\im{\mathrm{im\ }}

In this section we construct for \(l>0\), \((\alpha_l x_{2l+p-1},\sum_{j=1}^{l}\beta^l_j x_{p+2l-j} x_{p+j-1}\partial^p)\) such that \bas
\rho((\alpha_l x_{2l+p-1},\sum_{j=1}^{l}\beta^l_j x_{p+2l-j} x_{p+j-1}\partial^1))\mathsf{N}=\tau_p^{2l}\partial^p.
\eas
\begin{lemma}\label{lem:otrans}
 Let \(l>0\). Then \(\tau_p^{2l}\partial^{p}\in\im\rho\). 
\end{lemma}

\begin{proof}
Recall that
\bas\label{eq:inv2a}
		\tau^{2l}_p&=&\sum_{j=0}^{2l} (-1)^{j} 
		\binom{n-p-(2l-j)}{j}
		\binom{n-p-j}{2l-j} x_{p+2l-j}x_{p+j},
			\\&=&2\sum_{j=0}^{l-1} (-1)^{j} 
		\binom{n-p-(2l-j)}{j}
  + (-1)^{l} {\binom{n-p-l}{l}}^2 x_{p+l}^2
  \\&=&
  2\binom{n-p}{2l} x_{p+2l}x_{p}
  -2(n-p-(2l-1))\binom{n-p-1}{2l-1} x_{p+2l-1}x_{p+1}
  \\&&+2\sum_{j=2}^{l-1} (-1)^{j} \binom{n-p-(2l-j)}{j}\binom{n-p-j} {2l-j} x_{p+2l-j}x_{p+j}
  + (-1)^{l} {\binom{n-p-l} {l}}^2x_{p+l}^2.
  \eas
We compute
\bas
\lefteqn{\rho(( \alpha_l x_{2l+p-1},\sum_{j=1}^{l}\beta^l_j x_{p+2l-j} x_{p-1+j}\partial^p) )(x_2\partial^1 +\cdots+x_n\partial^{n-1})}&&\\
&=&( \alpha_l x_{p+1} x_{2l+p-1} -\sum_{j=1}^{l} \beta^l_j x_{p+1+2l-j} x_{p-1+j}-\sum_{j=1}^{l} \beta^l_j x_{p+2l-j} x_{j+p})\partial^p
+\sum_{j=1}^{l}\beta^l_j x_{p+2l-j} x_{p-1+j}\partial^{p-1}
\\&=&( \alpha_lx_{p+1} x_{2l+p-1}  -\sum_{j=0}^{l-1} \beta^l_{j+1} x_{p+2l-j} x_{j+p}-\sum_{j=1}^{l} \beta^l_{j} x_{p+2l-j} x_{j+p})\partial^p
+\sum_{j=1}^{l}\beta^l_j x_{p+2l-j} x_{p-1+j}\partial^{p-1}
\\&=&( \beta^l_{1} x_{p+2l} x_{p}+(\alpha_l -\beta^l_1-\beta^l_2)x_{p+1} x_{2l+p-1}    -\sum_{j=2}^{l-1} (\beta^l_{j+1}+\beta^l_{j})x_{p+2l-j} x_{j+p}-\beta^l_{l} x_{l+p}^2)\partial^p
\\&&+\sum_{j=1}^{l}\beta^l_j x_{p+2l-j} x_{p-1+j}\partial^{p-1}.
\eas
We ignore the \(\partial^{p-1}\)-term, as it is considered higher order in the maxtropical filtration topology.
So we need to solve
\bas
\alpha_l&=& \beta^l_1+\beta^l_2-2(n-p+1-2l ){\binom{n-p-1} {{2l}-1}} ,\\
\beta^l_1 &=&-2\binom{n-p}{2l} ,\\
\beta^l_j&=&-\beta^l_{j+1}-2(-1)^{j} 
    \binom{n-p-(2l-j)}{j}
    \binom{n-p-j} {2l-j},
    \quad j=2,\ldots,l-1,\\
\beta^l_{l}&=& -(-1)^l
    {\binom{n-p-l} {l}}^2,
\eas
and we see that this can be done, starting from \(\beta^l_l\) and going upward.
 \end{proof}
 \begin{theorem}[Quadratic outer normal form]\label{thm:ndonf}
  We can achieve a quadratic outer formal normal form for Equation \eqref{eq:eqs} through a sequence of invertible transformations. This normal form is given, modulo terms of degree \(>2\), by the following expression:
	\bas 
		v^{[2]}:=\N+
		\sum_{p=1}^n a_p x_p^2\frac{\partial }{\partial x_p}+\cdots.
	\eas
In other words, \(\mathfrak{n}^p_1=\langle x_p^2\partial^p\rangle, p=1,\ldots,n\) in the triangular case.
\end{theorem}
 \begin{proof}
     The proof follows from Theorem \ref{thm:ndnf} and Lemma \ref{lem:otrans}.
 \end{proof}
\subsection{Low dimensional examples}	
In this section, we will use the outer normal form provided in the previous section to determine the finite determinacy around the origin of \(2\)D feed forward and \(3\)D feed forward systems. Through a few steps, we will prove Theorems \ref{thm:2dtrun} and \ref{thm:final3DNF}. The authors of \cite{gandhi2020bifurcations} applied singularity theory to determine the finite determinacy of this 
\(2\)D network, arriving at the same results as ours. See also \cite{gazor2013normal,algaba2003quasi,algaba2004algorithm} for related discussions. For further details on these topics, we refer the reader to \cite{gazor2013normal,baider1991unique,baider1992further}.
\begin{enumerate}
    \item[I.]  In the first step, one should find the first-level normal form as recalled in Section \ref{sec:preli}.
 In the first-level normal form, only the linear part contributes, while in further reductions, we also consider the nonlinear terms. We use \(\delta^{(1)},\) for the classification in the normal form of the first level. 

\item [II.]
 
After finding the first-level normal form, one proceeds to further reduce it. This method involves using the nonlinear terms (very first) in the first-level normal form along with the linear part to achieve additional simplification, if possible. For example for \(2\)D case
the leading term is  \(\mathcal{L}_1,\) see \eqref{eq:L12d}.  The new grading should then be defined to make these terms of the same grade, and we serve \(\delta^{(2)}\) for the grading of the second level normal form. Note that for our feed forward networks, the second level normal form is the unique form as well.

\item [III.]
In the outer normal form, we use the ring of polynomials with coordinate transformations to simplify the first-level normal form, as we define the \(\rho\)-operator in Theorem \ref{thm:oop}. 
For the \(2\)D case, see Equation \eqref{EQ: rho2}.  Note that, to find the outer normal form,  we combine items II and III.
This is the operator used in \eqref{eq:RHOO2D} for \(2\)D  and \eqref{eq:r3} for \(3\)D.
\end{enumerate}

We would like to note that, in general, there is no specific way to find the unique normal form. However, the main idea is to use the terms in the first-level normal form, along with the linear part, to make further simplifications if possible.

 \subsubsection{$2$D outer versal normal form}\label{sec:2dorbital}
In this section, we will apply the operator 
\(\rho\)
as defined in Theorem \ref{thm:oop} to the 
\(2\)D Equation \eqref{eq:n2}. Our goal is to prove the following theorem:

  \begin{theorem}[Outer unique normal form]\label{thm:2dtrun}
		There is a formal  transformation 
  that transforms 
   \bas
			\left\{ \begin{array}{l} \dot {x}_{1} = 
				{x_{2}}+ F_1({x_{1}},{x_{2}}),\\ 
				\dot {x}_{2} =F_2({x_{2}}),
			\end{array} \right.
   \eas
 to the following outer formal normal form 
 \bas
	\left\{ \begin{array}{l} \dot {x}_{1} = x_2+
		\left(\mu_1+\sum_{n=2}^{\infty}\sum_{k=0}^{n}
\left( \frac{-1}{2a_2} \right)^k c_{n-k}
     \mu_1^{n} \right)x_1+a_2x_1^2,\\ 
		\dot {x}_{2}  =\nu+\mu_2 {x_{2}}+\bar b_2 x_2^2. \end{array} \right.
  	\eas

   where \(\mu_1,\mu_2\) and \(\nu\) are versal parameters and  \(a_2,\bar b_2\) are supposed to be invertible.
   	\end{theorem} 
  \begin{itemize}
  \item 
   \textbf{Step (1):}  In the first step, we find the first level normal form up to any grade. 
  The general first-level  normal form of \eqref{eq:n2} is given by
 \ba\label{eq:n2s}
	\left\{ \begin{array}{l} \dot {x}_{1} = 
		\mu_1 {x_{1}}+{x_{2}}+a_2 x_{1}^2+\sum_{i=3}^{\infty} a_i x_{1}^i,\\ 
		\dot {x}_{2}  =\nu+\mu_2 {x_{2}}+b_2 x_{2}^2+\sum_{i=3}^{\infty} b_i x_{2}^i. 
  \end{array} \right.
  	\ea
where  \(\nu,\mu_1\) and \(\mu_2\) are versal parameters and \(a_i\) and \(b_i\) are the coefficients of the normal form which are functions of the coefficients of the original system.  

We use the standard grade \(\delta^{(1)}(x_1^ix_2^j\frac{\partial }{\partial {x_{1}}})=i+j-1,\) for finding the first level normal form, see  Section \ref{sec:preli}.
We can observe that \([{x_{2}} \frac{\partial }{\partial {x_{1}} }, x_1^ix_2^j \frac{\partial }{\partial {x_{1}}}]= ix_1^{i-1}x_2^{j+1}  \frac{\partial }{\partial {x_{1}}}.\) Therefore, the terms of the form \(x_1^{i}\) produce the normal form, which is also in the \(\ker(\mathsf{L}_{\M})\), see \eqref{eq:sl2v2}.

      \item  \textbf{Step (2):}\label{step:22}
      Using a near-identity orbital transformation \(1+\sum_{i=1}^{\infty} \zeta_i x_2^i\)  one can eliminate the higher-order terms in the second component, except for the term \(x_2^2\).

  We apply this transformation as

  \(\left(1+\sum_{i=1}^{\infty} \zeta_i x_2^i \right) \left( \nu+\mu_2 {x_{2}}+b_2 x_2^2+\sum_{i=3}^{\infty} b_i x_2^i\right).\) We find 
   \bas
   \left\{ \begin{array}{l} \dot {x}_{1} = 
		\mu_1 {x_{1}}+{x_{2}}+a_2x_1^2 +\sum_{i=3}^{\infty} a_i x_1^i,\\ 
\dot {x}_{2}  =\nu+\mu_2 {x_{2}}+\bar{b}_2 x_2^2. \end{array} \right.
   \eas
\item {\bf Step (3):}
Set 
\(\mathsf{a}^i_j:=x_1^{i+1}x_2^j\frac{\partial }{\partial  {x_{1}}},
\)
and define the new grading  by
	$	\delta^{(2)}(\mu_1^{{p}}{\mathsf{a}}^i_j):={p}+i+2j$.
  For simplicity in the notations  rewrite the first component of \eqref{eq:n2s} to the following form
	\ba\label{eq:n2snb}
v^{[1]}:=\mathcal{L}_{1}+\sum_{i=1}^{\infty}a_i\mathsf{a}^i_0=\mathcal{L}_{1}+\sum_{i=2}^{\infty}v_i^{[1]},
	\ea
	where 
	\ba\label{eq:L12d}
 \mathcal{L}_{1}:=(x_2+\mu_1 {x_{1}}+a_2x_{1}^2)\frac{\partial }{\partial  {x_{1}}}=\mu_1\mathsf{a}^0_0
 +a_2\mathsf{a}^1_0+\mathsf{a}^{-1}_1, \,\,\, \delta^{(2)}(\mathcal{L}_1)=1.
	\ea

Set  \(\mathfrak{t}^p\) to be the graded vector space \(\mathfrak{t}^p = \sum_{n=1}^{\infty} \mathfrak{t}^p_n\), and \(\mathfrak{r}^p = \bigoplus_{i=1}^{\infty} \mathfrak{r}^p_i\) is the ring of orbital transformations, see \eqref{eq:space} and \eqref{eq:space2}.
 Hence,
	
 \bas
\mathfrak{r}^2=\left\{\sum_{0\leq m,k}c_{m,k}x_{1}^m x_2^k=\sum_{0\leq m,k}c_{m,k}\mathsf{o}^m_{k}\mid  c_{0,0}=0,c_{m,k} \in \mathbb{R}\right\}.
	\eas
 Since we are going to use the transformation for the first component of \eqref{eq:n2s} and it is only a function of \(x_1\), we define \(\mathsf{o}^m_0 = \mathsf{o}^m = x_1^m\), with \(\delta_0(\mathsf{o}^m) = m\).
The structure constants are as follows:

 \bas
	\mathsf{o}^m\mathsf{o}^n&=&\mathsf{o}^{m+n},
	\\
	\mathsf{o}^m\mathsf{a}^i_j&=&\mathsf{a}^{i+m}_j,
	\\
	{[\mathsf{a}^i_j,\mathsf{a}^l_k]}&=&(l-i) 	\mathsf{a}^{i+l}_{k+j}.
	\eas

We are now going to multiply \(a_2\mathsf{a}^{1}_0\) with \(x_1^n\). Then we use the Lie bracket to remove the terms \(x_2x_1^n \frac{\partial}{\partial x_1}\).
In fact,
 \bas
	\mathsf{o}^{n}\mathcal{L}_{1}&=&\mathsf{a}^{n-1}_1+\mu_1\,\mathsf{a}^{n}_0+a_2\mathsf{a}^{n+1}_0,\\
	{ [\mathsf{a}^{n}_0,\mathcal{L}_{1}]}&=&-n\mathsf{a}^{n-1}_1
	\,-\mu_1 (n-1)\,\mathsf{a}^{n}_0 -a_2
 (n-2)  \mathsf{a}^{n+1}_0.
	\eas
\item
 {\bf Step (4):} 
 Now, we apply the \(\rho\) operator to \eqref{eq:n2snb} inductively. Assume that for all grades less than \(n-1\), we find the normal which is 
 \(\mathcal{L}_{1}+\sum_{i=2}^{n-1} p_i\mu_1^i x_1 \frac{\partial }{\partial {x_{1}}}\)
 or equivalently,
 \ba\label{eq:nfinf2}
	\left\{ \begin{array}{l} \dot {x}_{1} = 
		\mu_1 x_1+a_2x_1^2+{x_{2}}+\sum_{i=2}^{n-1}p_i\mu_1^i x_1+\sum_{k=n}^{\infty}\sum_{i=0}^{k} c_i\mu_1^{i} x_{1}^{k-i+1},
   \\ 
		\dot {x}_{2}  =\nu+\mu_2 {x_{2}}+\bar b_2 x_2^2. \end{array} \right.
  	\ea
 We will now proceed to find the normal form of grade \(n\). 
  Take the vector field with  \(n\)-th grade  from \eqref{eq:nfinf2}
  \bas
			v^{[1]}_{n}=\sum_{i=0}^{n} c_i\mu_1^{i} x_{1}^{n-i+1}
   \frac{\partial }{\partial {x_{1}}}=\sum_{i=0}^{n} c_i\mu_1^{i}
   \mathsf{a}^{n-i}_0
   \,\,\,\,\,\,\qquad\delta^{(2)}(v_n^{[1]})=n.
			\eas
 Define the coordinates and orbital transformations as 
   	\ba\label{eq:nftrans1}
			\mathsf{o}^{n-1}&:=&\sum_{i=0}^{n-1} \alpha_i\mu_1^i x_{1}^{n-i-1}=\sum_{i=0}^{n-1} \alpha_i\mu_1^i\mathsf{o}^{n-i-1},\,\,\,\,\,\,\,\,\,\,\,\qquad
\delta_0 (\mathsf{o}^{n-1})=n-1,
   \\\label{eq:nftrans2}
  \mathsf{c}^{n-1}&:=&\sum_{i=0}^{n-1}\beta_i\mu_1^i x_{1}^{n-i}\frac{\partial }{\partial {x_{1}}}=
\sum_{i=0}^{n-1}\beta_i
\mu_1^i 
\mathsf{a}^{n-i-1}_0,\quad\quad\quad
\delta^{(2)}( \mathsf{\mathsf{c}^{n-1}})=n-1.
			\ea

   We use the \(\rho\) operator, cf. Theorem \ref{thm:ro}. Hence, we get
      \ba\label{eq:RHOO2D}
       && \rho( \mathsf{o}^{n-1},\mathsf{c}^{n-1} )(\mathcal{L}_1+v^{[1]}_{n})=
   [\mathsf{c}^{n-1} ,v^{[1]}_{n}+\mathcal{L}_1]+\mathsf{o}^{n-1}(\mathcal{L}_1+v^{[1]}_{n})+\mathcal{L}_1+ v^{[1]}_{n}.
      \ea
      Since we aim to simplify the term of grade \(n\) then the focus would be on the terms that generate this grade: 
			\ba\label{EQ: rho2}
   v^{[1]}_{n}+\rho(\mathsf{o}^{n-1},\mathsf{c}^{n-1} )\mathcal{L}_1&=&
   [\mathsf{c}^{n-1} ,\mathcal{L}_1]+\mathsf{o}^{n-1} \mathcal{L}_1+ v^{[1]}_{n}
   \\\nonumber
&=&
\sum_{i=0}^{n-1}
\beta_{i}  \mu_1^{i} 
\left( \left(i-n\right)\mathsf{a}^{n-2-i}_1+\mu_1  \left(i -n +1\right) \mathsf{a}^{n-i-1}_0+a_{2} \left(i -n +2\right) \mathsf{a}^{n-i}_0\right)
\\\nonumber
&&+
\sum_{i=0}^{n-1} \alpha_i\mu_1^i \left(\mathsf{a}^{n-2-i}_1+
\mu_1 \mathsf{a}_0^{n-i-1}+
a_2\mathsf{a}_0^{n-i}\right)
+\sum_{i=0}^{n-1} c_i\mu_1^{i} \mathsf{a}^{n-i}_0+
c_n\mu_1^{n} \mathsf{a}^0_0\\\nonumber
&=&
\sum_{i=0}^{n-1}\left(  
\beta_{i}  \left(i-n\right)
+\alpha_i\right) \mu_1^i \mathsf{a}^{n-2-i}_1
\\&&+\nonumber\sum_{i=1}^{n-1}\left(\beta_{i-1} \left(i -n \right) +\alpha_{i-1}+a_2\left( \alpha_i+\beta_{i}\left(i -n+2\right)\right)+ c_i\right)
\mu_1^{i} \mathsf{a}^{n-i}_0
\\&&+\left( a_2\left(\alpha_0+ \beta_{0}\left( -n+2\right)\right)+ c_0\right)
 \mathsf{a}^{n}_0+{\left(
\alpha_{n-1}+c_n\right)\mu_1^{n} \mathsf{a}^0_0}.
\ea
Putting in \(\beta_{i}  \left(i-n\right)+\alpha_i=0, i=0,\ldots,n-1\),
we obtain
\bas
 v^{[1]}_{n}+\rho(\mathsf{o}^{n-1},\mathsf{c}^{n-1} )\mathcal{L}_1&=&
\sum_{i=1}^{n-1}\left(\beta_{i-1}+a_2\left( 2\beta_{i}+ c_i\right)\right)
\mu_1^{i} \mathsf{a}^{n-i}_0
+\left( 2 a_2\beta_{0}+ c_0\right)
 \mathsf{a}^{n}_0+{\left(
\beta_{n-1}+c_n\right)\mu_1^{n} \mathsf{a}^0_0}
\\&&
+\left(\beta_{0}+a_2\left( 2\beta_{1}+ c_1\right)\right)
\mu_1 \mathsf{a}^{n-1}_0.
\eas
So we need to solve \(\beta_{i-1}+a_2\left( 2\beta_{i}+ c_i\right)=0,\) for \( i=1,\ldots,n-1\) with \(\beta_0=-\frac{c_0}{2a_2}\).

  We find
  \begin{equation}
\begin{cases}
  \beta_i  = \sum_{k=1}^{i+1} \left( -\frac{1}{2a_2} \right)^k c_{i-k+1} & \text{for } 0 \leq i \leq n-1, \\
  \alpha_i = (n-i) \beta_i, & \text{for } 0\leq i \leq n-1.
\end{cases}
\end{equation}
Hence, \(\alpha_{n-1}=\beta_{n-1}=\sum_{k=1}^{n} \left( \frac{-1}{2a_2} \right)^k c_{n-k}\) and
\bas
 v^{[2]}_n(x_1)&=& v^{[1]}_{n}+\rho(\mathsf{o}^{n-1},\mathsf{c}^{n-1})\mathcal{L}_1=
   (\alpha_{n-1}+c_n)\mu_1^{n} \mathsf{a}^{0}_0
\\&=&
   \sum_{k=0}^{n}
\left( \frac{-1}{2a_2} \right)^k c_{n-k}
     \mu_1^{n} \mathsf{a}^{0}_0=\mathcal{V}_n^{[2]}\mathsf{a}^0_0.
\eas

Since 
\bas
			v^{[1]}_{n}(x_1)&=&\sum_{k=0}^{n} c_{n-k}\mu_1^{n-k} x_{1}^{k+1}
   \frac{\partial }{\partial {x_{1}}}=\sum_{k=0}^{n} c_{n-k}\mu_1^{n-k}
   x_1^k \mathsf{a}^{0}_0=\mathcal{V}_n^{[1]}(x_1)\mathsf{a}^0_0,
			\eas
 one has
 \bas
 \mathcal{V}^{[1]}_n(\frac{-\mu_1}{2a_2})&=&\sum_{k=0}^{n} c_{n-k}\left(\frac{-\mu_1}{2a_2 }\right)^{k}\mu_1^{n-k}
   \\&=&\sum_{k=0}^{n} c_{n-k}\left(\frac{-1}{2a_2 }\right)^{k}\mu_1^{n}
  \\&=&\mathcal{V}^{[2]}_n.
  \eas

\end{itemize}
  
 \subsubsection{$3$D outer unique normal form}\label{sec:3dorbital}
In the final part of this paper, we focus on the analysis of $3$D feed forward systems. Our methodology involves the application of outer normal forms, analogous to the techniques employed for $2$D systems in the previous section. However, in this case, we do not involve versal parameters. We aim to prove the following theorem:

\begin{theorem}[Outer unique normal form]\label{thm:final3DNF}
For the {\em $3$D} feed forward network described by the following system near the triple-zero singular fixed point:
 \begin{align}\label{eq:3dd}
      \begin{cases} 
		\dot{x}_1 &= {x_{2}} + F_1({x_{1}},{x_{2}},{x_{3}}), \\
		\dot{x}_2 &= {x_{3}} +F_2({x_{2}},{x_{3}}), \\
		\dot{x}_3 &= F_3({x_{3}}),
	  \end{cases}
  \end{align}
 there exists  an invertible feed forward transformations that bring it to the following outer unique  normal form:
 
	\begin{align}\label{eq:3dl1a}
	\begin{cases} 
		\dot {x}_{1} &= {x_{2}} + a_2 x_1^2, \\
		\dot {x}_{2} &= {x_{3}} + b_2x_{2}^2, \\
		\dot {x}_{3} &= c_2x_3^2.
	\end{cases}
\end{align}

\end{theorem}

We begin by establishing some notation before starting the proof. Set the ring of the orbital transformation to be 
\begin{align*}
\mathsf{T}_1 &:=
\left\{\sum_{m,n,s\in \mathbb{N}_0}\alpha^m_{n,s} x_{1}^m x_{2}^n {\Delta}^s \mid  \alpha^{0}_{0,0}=0,\alpha^{m}_{n,s} \in \mathbb{R}\right\}, 
\end{align*}
and \(\Delta = 2{x_{1}}{x_{3}} - x_{2}^2\), which is the second transvectant of \({x_{1}}\) with itself; see Section \ref{subsec:examples}.
 Define \(\delta_0(x_{1}^mx_{2}^n{\Delta}^s)=m+n+2s\) and \(\delta^{(1)}(x_{1}^ix_{2}^j{\Delta}^k
	\frac{\partial}{\partial {x_{1}}})=i+j+2k-1\)
 which are the grades for the first level normal form.
 Using formula \eqref{eq:inductivesl2}, the \(\Sl\)-triple of the \(3\)D nilpotent system (corresponding to the linear part of \eqref{eq:3dd}) is:
  \begin{align*}
	\N &= {x_{2}}\frac{\partial}{\partial {x_{1}}} + {x_{3}}\frac{\partial}{\partial {x_{2}}}, \\
	\M &= 2{x_{1}}\frac{\partial}{\partial {x_{2}}} + 2{x_{2}}\frac{\partial}{\partial {x_{3}}}, \\
	\H &= 2{x_{1}}\frac{\partial}{\partial {x_{1}}} - 2{x_{3}}\frac{\partial}{\partial {x_{3}}}.
  \end{align*}
	The structure constants  are given by 
 \bas
 [x_{1}^{m} x_{2}^{n} \Delta^{s}\frac{\partial}{\partial {x_{1}}},x_{1}^{i} x_{2}^{j} \Delta^{k}\frac{\partial}{\partial {x_{1}}}]&=&\left( \left(k -s \right) x_{2}^{n +j +2}\Delta^{s +k -1}+ \left(i +k -m -s \right) x_{2}^{n +j}\Delta^{s +k} \right) x_{1}^{m +i -1} \frac{\partial}{\partial {x_{1}}},
 \\
 {[x_{1}^{i} x_{2}^{j} \Delta^{k} \frac{\partial}{\partial {x_{1}}}, \N]}&=&
- \frac{1}{2}\left(jx_{1}^{i -1} x_{2}^{j -1} \Delta + \left(2 i +j \right)x_{2}^{j +1} \right) x_{1}^{i-1} \Delta^{k}\frac{\partial}{\partial {x_{1}}},
 \\
 x_{1}^{m} x_{2}^{n} \Delta^{s}\cdot x_1^{i} x_2^{j} \Delta^{k}\frac{\partial}{\partial {x_{1}}}
 &=&
 x_{1}^{m+i} x_2^{n+j} \Delta^{s+k}\frac{\partial}{\partial {x_{1}}}.
 \eas 
 \begin{itemize}
 \item{\textbf{Step (1):} }The initial step involves deriving the first level  normal form for \eqref{eq:3dd}, up to any grade, using the triangular  \(\Sl\)-style. 
  \begin{lemma}
  \label{thm:tri3}
  There exists a sequence of formal feed forward  coordinate transformations that bring the system  \eqref{eq:3dd}
  to the normal form:
  \begin{align}\label{eq:unf}
      \begin{cases} 
		\dot{x}_1 &= {x_{2}} + \bar{f}_1({x_{1}}, \Delta), \\
		\dot{x}_2 &= {x_{3}} + \bar{f}_2({x_{2}}), \\
		\dot{x}_3 &= F_3({x_{3}}),
	  \end{cases}
  \end{align}
  where \(\bar{f}_1\) and \(\bar{f}_2\) are nonlinear polynomials.
\end{lemma}
  
  \begin{proof}
  Before presenting the proof, we remark that the proof is similar to the first part of the normal form for the \(3\)D system discussed in \cite[Chapter 12.6.2]{sanders2007averaging}.
 Following our \(\Sl\) triangular normal form, we put all components in the normal form. For the third component, there is nothing to remove, since \(\Sl\) is zero. The second component
 \({\ker}(\mathsf{L}_{{\M}}) \mid \mathfrak{r}_2\) is generated by \({x_{2}}\), which is the only generator term in the normal form for any order (see Example \eqref{eq:nf2d}). 

  \({\ker}(\mathsf{L}_{{\M}})\mid \mathfrak{r}_3\) is generated by \(\mathcal{I} = \langle \mathsf{a}, \mathsf{b} \rangle\), where \(\mathsf{a} = {x_{1}}\) and \(\mathsf{b} = (\mathsf{a}, \mathsf{a})^{(2)} = 2{x_{1}}{x_{3}} - x_2^2\), which is the second transvectant of \(\mathsf{a}\) with itself (see Lemmas \eqref{lem:trans} and Equation \eqref{eq:nf3}).
  Corresponding to each element with the monomial \(u^{\omega}t^{\beta}\), where \(\omega\) is the weight of the kernel and \(\beta\) is its degree, we have \(\omega(\mathsf{a}) = 2\) and \(\omega(\mathsf{b}) = 0\). The generating function is given by:
  \begin{align*}
	P^3(t, u) = \frac{1}{(1 - u^2t)(1 - t^2)}.
  \end{align*}
  Now, we apply the Cushman-Sanders test \cite{cushman56nilpotent,cushman1990survey} to show that \(\mathcal{I}\) indeed generates \( {\ker}(\mathsf{L}_{\M})\). Taking the derivative of \(uP^3(t,u)\) with respect to \(u\) and putting \(u = 1\), the result is:
  \begin{align*}
	\frac{\partial }{\partial u } (uP^3(t,u))\mid_{u=1} = \frac{1}{(1-t)^3},
  \end{align*}
  which gives us the generating function of polynomials with three variables. See Section 12 of \cite{sanders2007averaging} for more details and examples. This shows that \(\mathcal{I} = \langle \mathsf{a}, \mathsf{b} \rangle\) generates the normal form of the first component of \eqref{eq:3dl1a}. 
 
  \end{proof}

 \item \textbf{Step (2):}
       Now, we start with Equation \eqref{eq:unf}. The aim is to employ the \(\rho\) action by carrying out a sequence of transformations. For the last component of \eqref{eq:unf}, we apply the orbital transformation given by the second step of \(2\)D case, and for the second component, we apply the coordinate and orbital transformations given in \eqref{eq:nftrans1} and \eqref{eq:nftrans2}. 
       
       Specifically, we begin with the following system:

\begin{align}\label{eq:3dl1b}
	\begin{cases} 
		\dot {x}_{1} &= {x_{2}} + \bar{F}_1({x_{1}},\Delta), \\
		\dot {x}_{2} &= {x_{3}} + b_2x_{2}^2, \\
		\dot {x}_{3} &= c_2x_3^2.
	\end{cases}
\end{align}
Note that after applying the orbital normal to the third and second components, the coefficients of the terms in the first component will be changed. This is why we replace \(\bar{f}_1({x_{1}},\Delta)\) with \(\bar{F}_1({x_{1}},\Delta)\).

Rewrite \(\dot {x}_{1} = {x_{2}} + \bar{F}_1({x_{1}},\Delta)\) as \(\dot {x}_{1} = {x_{2}} + \sum_{p=2}^{\infty} v_p^{[1]}\) where \(v_p^{[1]} = \sum_{k=0}^{\lfloor{\frac{p}{2}}\rfloor} a^{p-2ki}_{k}x_1^{p-2k}\Delta^{k}\)  (with \(a^{p-2ki}_{k}\) as coefficients) with \(\delta^{(1)}(v_p^{[1]}) = p-1\).

\item \textbf{Step (3):}
By induction, we will prove that all higher-order terms can be eliminated by orbital and coordinate transformations using the action \(\rho\) given in \eqref{eq:RHOOP}. Thus, we assume that we have removed all terms of \(v_p^{[1]}\) for all \(p = 1, \ldots, n-1\). Now, we apply the orbital normal form for grade \(n = l + 2m\).

For the given grade, we have two types of terms: \(x_1^l \Delta^m\) with \(m > 0\) and
\(m = 0\). We remove these terms in the following order:

\begin{align*}
	x_1^{p-2m}\Delta^{m} &\prec x_1^{p-2m+2}\Delta^{m-1} \prec \cdots \prec x_1^p,
\end{align*}

in fact, we start to remove the terms with \(\Delta\) highest power in one level and the term with \(m = 0\) in the next level.
 
\begin{lemma}\label{lem:o31}
There  are two class of transformation \((\mathfrak{r},\mathfrak{t})\) generated by:
\begin{align*}
	\mathsf{o}^l_m &= -(2l+3) a^l_m {x_{2}}  x_1^l\Delta^{m-1},\\
	\mathsf{c}^l_m &= 2a^l_m{x_{2}}x_1^{l} \Delta^{m-1} x_1\frac{\partial}{\partial {x_{1}}},
\end{align*}
that eliminates terms  in type \(x_1^{l}\Delta^{m} \frac{\partial}{\partial x_1}\)  where \(m>0\) from \(\dot{x}_1\) of  \eqref{eq:3dl1b}, where  $a^l_m$  is the coefficient of  $x_1^l\Delta^{m}\frac{\partial}{\partial {x_{1}}}$.
\end{lemma}

\begin{proof}

Define
\begin{align*}
	\mathsf{o}^l_m &= \alpha^l_m {x_{2}} x_1^l\Delta^{m-1},\quad m > 0, l\geq 0\\
	\mathsf{c}^l_m &= \beta^l_m {x_{2}}x_1^{l} \Delta^{m-1} x_1\frac{\partial}{\partial {x_{1}}},
\end{align*}
where \(\alpha^l_m\) and \(\beta^l_m\) are coefficients that need to be determined. By applying the \(\rho\) operator \eqref{eq:RHOOP}, we find the following results.

\ba
\nonumber
	\rho(\mathsf{c}^l_m,\mathsf{o}^l_m){x_{2}}\frac{\partial}{\partial {x_{1}}}&=&[\mathsf{c}^l_m,{x_{2}}\frac{\partial}{\partial {x_{1}}}]+ \mathsf{o}^l_m {x_{2}}\frac{\partial}{\partial {x_{1}}}
 \\\nonumber&=&
 \alpha^l_m\left(2 {{x_{3}}} x_1^{l+1} \Delta^{m-1}- x_1^{l}\Delta^{m}\right)\frac{\partial}{\partial x_{1}}
 -\left(2 l +3\right) \beta^l_m {{x_{3}}} x_1^{l+1}\Delta^{m-1} \frac{\partial}{\partial x_{1}}
 \\\label{eq:r3}
	&=&  \left(2 \alpha^l_m -(2 l+3)\beta^l_m  \right){x_{3}} x_1^{l+1} \Delta^{m-1}\frac{\partial}{\partial x_{1}}
 - \left(\alpha^l_m-(l+1)\beta^l_m \right)x_1^{l}\Delta^{m}\frac{\partial}{\partial x_{1}}.
\ea
Define:
\bas
	\alpha^l_m &=&\frac{2 l+3}{2}\beta^l_m, \\
	\beta^l_m &=& -2 a^l_m.
\eas
Then we find 
\bas
[\mathsf{c}^l_m,{x_{2}}\frac{\partial}{\partial {x_{1}}}]+ \mathsf{o}^l_m {x_{2}}\frac{\partial}{\partial {x_{1}}}=a^l_m{ x_1^{l}\Delta^{m}}\frac{\partial}{\partial x_{1}}.
\eas
Therefore, \( x_1^{l} \Delta^{m} \frac{\partial}{\partial x_{1}} \in {\rm{Im}}(\rho)\)  for \(l\geq 0, m>0\).
\end{proof}

Assume that \(a_2=a^2_0\neq 0\), therefore \(x_1^{2}\) is the first term which is non-zero in the unique triangular \(\Sl\)  normal form given by Theorem \ref{thm:tri3}. Then we find the following proposition follows from Lemma \ref{lem:o31}.

\item \textbf{Step (4):}  In the last step we aim to eliminate all terms from \eqref{eq:3dl1b}, except \(({x_{2}}+a_2 x_1^{2})\frac{\partial}{\partial {x_{1}}},\)
for more details related to the higher-order normal form or unique normal form, see \cite{baider1991unique,baider1992further,gazor2013normal,gazor2019vector,gazor2013volume}.
 In the following, we update our grade so that the linear part and \(x_1^{2}\) have the same grade.
  Define \(\delta^{(2)}(x_i)=s_i, \delta^{(2)}(\frac{\partial}{\partial x_i})=-s_i\) for all \(i=1,2,3.\)
 Therefore, $$\delta^{(2)}({x_{2}}\frac{\partial}{\partial {x_{1}}})=s_2-s_1, \,\,\,\delta^{(2)}(x_1^{2}\frac{\partial}{\partial {x_{1}}})=s_1,\,\,  \mbox{and}\,\,\,
 \delta^{(2)}({x_{3}}\frac{\partial}{\partial {x_{2}}})=s_3-s_2.$$
 Since we would like to have the same grade for terms of  linear part (\(\N\)) and 
 \({x_{2}}\frac{\partial}{\partial {x_{1}}}\) and 
 \(x_1^{2}\frac{\partial}{\partial {x_{1}}}.\) 
 We have to find \(s_i\) such that the following hold true

\bas
\delta^{(2)}({x_{2}}\frac{\partial}{\partial {x_{1}}})&=&\delta^{(2)}({x_{3}}\frac{\partial}{\partial {x_{2}}}),\,\,\,\ \mbox{two terms of}\,\,\, \N,
\\
\delta^{(2)}({x_{2}}\frac{\partial}{\partial {x_{1}}})&=&\delta^{(2)}(x_1^{2}\frac{\partial}{\partial {x_{1}}}),\,\,\, \mbox{leading terms of first component,}\eas
above equations give rise to
\bas
s_3-s_2&=&s_2-s_1,
\\s_2&=&2s_1,
\eas
we find $s_3=3s_1.$ Now, set \(s_1=1\) we get \(s_2=2\) and \(s_3=3.\) Therefore, 
\bas
	\delta^{(2)}(x_1^{i}x_2^jx_3^k\frac{\partial}{\partial x_l})=&i+2j+3k-s_l\,\,\mbox{for all},\,\, l=1,2,3.
\eas

Now, define 
	\ba\label{eq:l1}
	\mathcal{L}_{1}&:=&\mathsf{N}+a_2 x_1^{2}\frac{\partial}{\partial {x_{1}}},
	\ea
	with $\delta^{(2)}(\mathcal{L}_{1})=1$. And we have that \(\delta^{(2)}(x_1^3\frac{\partial}{\partial {x_{1}}})=2>1\)
 and \(\delta^{(2)}(\Delta\frac{\partial}{\partial {x_{1}}})=3>1\).

\end{itemize}
\begin{lemma}
Assume \(a_2 x_1^2 \frac{\partial}{\partial {x_{1}}}\) to be the lowest
{\em(in  \(\delta^{(2)}\)}-grade{\em )} non zero terms in the unique inner normal form given by Theorem \eqref{thm:tri3}. The outer normal form of the first component of \eqref{eq:3dd}  is 
    \bas
    {\dot x}_1={x_{2}}+a_2 x_1^{2}.
    \eas
\end{lemma}
\begin{proof}
Consider these transformation generators
\bas
 \mathsf{o}^l_0&:=&\alpha^l_0x_1^{l-2},\quad l>2,
 \\
\mathsf{c}^l_0&:=& \beta^l_0 x_1^{l-1} \frac{\partial}{\partial {x_{1}}}.
 \eas
 By carrying these transformations through the \(\rho\) action we get 
 \bas
\rho(\mathsf{o}^l_0,\mathsf{c}^l_0)\mathcal{L}_{1}&=&
 \mathsf{o}^l_0 \mathcal{L}_{1} +[\mathsf{c}^l_0,\mathcal{L}_{1}]
   \\
   &=&
 \left(-( l -1)\beta^l_0   +\alpha^l_0 \right) {{x_{2}}}  x_1^{l -2} \frac{\partial}{\partial {x_{1}}}+\alpha^l_0 a_2 \, x_1^{l}  \frac{\partial}{\partial {x_{1}}}.
 \eas
By choosing 
 \bas
 \beta^l_0&=&\frac{\alpha^l_0}{l -1},\quad
 \alpha^l_0 = \frac{a^l_0}{a_2},
 \eas
we get  \(
  \rho((\mathsf{o}^l,\mathsf{c}^l))\mathcal{L}_{1}=
a^l_0 x_1^l\frac{\partial}{\partial {x_{1}}}\).  
 Hence  \({ x_1^{l}}\frac{\partial}{\partial {x_{1}}} \in {\rm{Im}}(\rho)\) for all \(l>2.\)
 By applying the normalization repeatedly, we can remove the terms in the triangular vector field to any grade we want,
 leaving us with the formal normal form as given in this Lemma.
\end{proof}
We have now obtained the promised (in Theorem \ref{thm:final3DNF}) normal form:
\begin{align}\label{eq:3dl1c}
	\begin{cases} 
		\dot {x}_{1} &= {x_{2}} + a_2 x_1^2, \\
		\dot {x}_{2} &= {x_{3}} + b_2x_{2}^2, \\
		\dot {x}_{3} &= c_2x_3^2.
	\end{cases}
\end{align}

\section*{Acknowledgement}
We sincerely thank Bob Rink for initiating our discussions on this problem and for the extensive conversations that followed. We are also grateful to Jan Sanders for his invaluable expertise and insightful feedback, which significantly enhanced the quality of our work. Additionally, we thank Nicolas Augier for providing helpful references for this paper.
	\bibliographystyle{plain}

 \bibliography{NET.bib}
\end{document}